\documentclass[12pt]{article}
\title{Operator level limit of the circular Jacobi $\beta$-ensemble}
\date{}
\author{Yun Li and Benedek Valk\'o}

\usepackage{units}
\usepackage{amsmath, amsthm, amssymb,stackengine,bbm}
\usepackage{graphicx}
\usepackage{amsmath, enumerate}
\usepackage{color}
\usepackage{hyperref}
\newtheorem{theorem}{Theorem}
\newtheorem{lemma}[theorem]{Lemma}
\newtheorem{proposition}[theorem]{Proposition}
\newtheorem{corollary}[theorem]{Corollary}

\newtheorem{fact}[theorem]{Fact}
\newtheorem{assumption}[theorem]{Assumption}

\DeclareMathOperator{\sech}{sech}

\DeclareMathOperator{\arccot}{arccot}

\theoremstyle{definition} % For roman text in the body
\newtheorem{definition}[theorem]{Definition}
\newtheorem{remark}[theorem]{Remark}

\newcommand{\eps}{\varepsilon}
\newcommand{\Z}{{\mathbb Z}}
\newcommand{\CC}{{\mathbb C}}
\newcommand{\RR}{{\mathbb R}}
\newcommand{\ZZ}{{\mathbb Z}}

\newcommand{\im}{\mathrm{i}}

\newcommand{\R}{{\mathbb R}}

\newcommand{\HH}{{\mathbb H} }
\newcommand{\ev}{{\mathbb{E}}}

\newcommand{\lstar}{{\raise-0.15ex\hbox{$\scriptstyle \ast$}}}

\newcommand{\ol}{\overline}

\theoremstyle{remark} % For an italic header, more subtle than definition style
\newcommand{\Sineb}{\operatorname{Sine}_{\beta}}

\newcommand{\Bessel}{\operatorname{Bessel}}
\newcommand{\BB}{\operatorname{Bess}}

\newcommand{\HPb}{\operatorname{HP}_{\beta, \delta}}

\newcommand{\Besselop}{\mathtt{Bessel}_{\beta,a}}
\newcommand{\Bessop}{\mathtt{Bess}_{\beta,a}}
\newcommand{\CJop}{\mathtt{CJ}_{n,\beta,\delta}}
\newcommand{\ROop}{\mathtt{RO}_{2n,\beta,a,b}}

\newcommand{\Huaop}{\mathtt{HP}_{\beta,\delta}}
\newcommand{\err}{\textup{err}}

\newcommand{\uu}{\mathfrak{u}}

\newcommand{\ttHP}{\mathtt{HP}}

\newcommand{\ttB}{\mathtt{B}}

\newcommand{\ttr}{\mathtt{r\,}}
\newcommand{\res}{\ttr}

\newcommand{\intr}{\mathfrak{t\,}}

\newcommand{\Dirop}{\mathtt{Dir}}

\definecolor{violet}{rgb}{0.8,0,0.2}

\newcommand{\ed}{\stackrel{d}{=}}
\newcommand{\cB}{{\mathcal B}}

\newcommand{\mat}[4]{\left( \begin{array}{cc}
		#1 & #2  \\
		#3 & #4  \\
	\end{array} \right)}

\newcommand{\dom}{\operatorname{dom}}

\newcommand{\tr}{\operatorname{tr}}
\newcommand{\ind}{\mathbf 1}
\newcommand{\tch}{\upsilon} % the time change function
\newcommand{\cH}{\mathcal{H}}
\newcommand{\cA}{\mathcal{A}}
\newcommand{\cE}{\mathcal{E}}

\newcommand{\benedek}[1]{\textcolor{red}{\texttt{#1}}}

\newcommand{\tl}{\widetilde}

\newcommand{\bside}{\noindent\textbf{\benedek{Begin side computation.}}
	\begin{footnotesize}}
	
	\newcommand{\eside}{\end{footnotesize}
	\noindent \textbf{\benedek{End side computation.}}}
\newcommand{\spec}{\operatorname{spec}}

\begin{document}
	\maketitle

\begin{abstract}
We prove an operator level limit for the circular Jacobi $\beta$-ensemble. As a result, we characterize the counting function of the limit point process via coupled systems of stochastic differential equations. We also show that the normalized characteristic polynomials converge to a random analytic function, which we characterize via the joint distribution of its Taylor coefficients at zero and as the solution of a stochastic differential equation system. 
We also provide analogous results for the real orthogonal $\beta$-ensemble.
\end{abstract}

%\section{Outline}
%
%\subsection{Intro}
%
%Introduce the models, the limiting operators. State the main theorems.
%
%\subsection{Random operators}
%
%
%\subsection{The general convergence result}
%
%
%\subsection{real orthogonal}
%
%\subsection{Hua-Pickrell}

\section{Introduction}

We study two families of finite point processes on the circle: the \emph{circular Jacobi $\beta$-ensemble} (CJ$\beta$E) and the  \emph{real orthogonal $\beta$-ensemble} (RO$\beta$E).

For a given integer $n\ge 1$, $\beta>0$, and $\delta\in \CC$ with $\Re \delta>-1/2$ the size $n$ \emph{circular Jacobi $\beta$-ensemble}  with parameters $\beta, \delta$ is the joint distribution of 
$n$ distinct points $\{e^{i \theta_1}, \dots, e^{i \theta_n}\}$ with $\theta_j\in [-\pi,\pi)$, where the joint density function of the angles $\theta_j$ is given by
\begin{align}
 p_{n,\beta, \delta}^{cj}(\theta_1, \dots, \theta_n)=  \frac{1}{Z_{n,\beta,\delta}^{cj}} \prod_{j<k\le n} \left|e^{i \theta_j}-e^{i \theta_k}\right|^\beta  \prod_{k=1}^n (1-e^{- i \theta_k})^{ \delta}(1-e^{ i \theta_k})^{ \bar\delta} \label{eq:PDF_cjacobi}, \qquad \theta_j\in [-\pi,\pi).
\end{align}
Here  $Z_{n,\beta,\delta}^{cj}$ is an explicitly computable normalizing constant (see e.g. Section 4.1 of \cite{ForBook}).
% \[
% Z_{n,\beta,\delta}^{cj}=(2\pi)^{n}\prod_{k=0}^{n-1}\frac{\Gamma(1+(k+1)\beta/2)\Gamma(1+\delta+\bar\delta+k\beta/2)}{\Gamma(1+\beta/2)\Gamma(1+\delta+k\beta/2)\Gamma(1+\bar\delta+k\beta/2)}.
% \]

We  write $\Lambda_n\sim \textup{CJ}_{n,\beta,\delta}$ to denote that the random set $\Lambda_n=\{\theta_1, \dots, \theta_n\}$ has  joint density given by \eqref{eq:PDF_cjacobi}.
This family of distributions extends several other named ensembles.  For $\beta=2$ the distribution  was studied by Hua \cite{Hua} and Pickrell \cite{Pickrell}, and is known as the Hua-Pickrell measure in the literature. For $\delta=0$ the distribution  is called the circular $\beta$-ensemble. In particular, when $\beta=2$ and $\delta=0$ we get  the circular unitary ensemble, which gives the joint eigenvalue distribution of an $n\times n$ Haar unitary matrix. For $k\in \ZZ_+$ with $\delta=\frac{\beta k}{2}$ the measure given by  \eqref{eq:PDF_cjacobi} can also be realized as a conditioned version of the size $n+k$ circular $\beta$-ensemble, conditioned to have $k$ points at $1$ (i.e.~$\theta=0$). See  \cite{BO01}, \cite{FN_2001},  Section 3.12 of \cite{ForBook}, and the references within for additional information about the ensemble.

The \emph{real orthogonal $\beta$-ensemble} is a  family of distributions describing an even number of points on the unit circle in a reflection symmetric configuration. If we parametrize the points as $\{\pm e^{i \theta_1}, \dots, \pm e^{i \theta_n}\}$ with $\theta_j\in (0,\pi)$ then the joint density for $(\theta_1, \dots, \theta_n)$ is given by
\begin{align}
&p_{n,\beta,a,b}^{o}(\theta_1,\dots, \theta_n)=    \frac{1}{Z_{n,\beta,a,b}^o} \prod_{j<k\le n} |\cos(\theta_j)-\cos(\theta_k)|^\beta\notag \\ &\hskip150pt \times\prod_{k=1}^n |1-\cos(\theta_k)|^{\frac{\beta}{2}(a+1)-1/2} |1+\cos(\theta_k)|^{\frac{\beta}{2}(b+1)-1/2}. \label{eq:PDF_ortho}
\end{align}
Here $\beta>0, a>-1, b>-1$ are real parameters, and $Z_{n,\beta,a,b}^o$  is an explicitly computable normalizing constant (see \cite{KK}). The ensemble was introduced in \cite{KillipNenciu} and \cite{KK} as a generalization of the joint eigenvalue distributions of some of the classical ensembles on the orthogonal and special orthogonal group of matrices.  E.g.~with $\beta=2$, $a=b=\frac{1}{\beta}-1$, we get the joint eigenvalue distribution of a $2n\times 2n$ special orthogonal matrix chosen  according to Haar measure on $SO(2n)$. (Note that our choice of parameters are slightly different from the one used in \cite{KillipNenciu} and \cite{KK}.)
We  write $\Lambda_{2n}\sim \textup{RO}_{2n,\beta,a, b}$ 
to denote that the random set $\Lambda_{2n}=\{\pm \theta_1, \dots, \pm\theta_n\}$ has a distribution determined by the joint density given by
\eqref{eq:PDF_ortho}.

The real orthogonal $\beta$-ensemble can be connected to another named ensemble on the real line via a change of variables. Suppose that the joint distribution of $\{\theta_1, \dots, \theta_n\}$ is determined by the density \eqref{eq:PDF_ortho}, and introduce \begin{align}\label{eq:changeofvarJ}
x_j=\tfrac12 (1-\cos(\theta_j)).
\end{align}
Then $\{x_1, \dots, x_n\}\in(0,1)^n$ has joint density function
\begin{align}\label{eq:PDF_Jacobi}
   \frac{1}{Z_{n,\beta,a,b}^j}
    \prod_{j<k\le n}|x_j-x_k|^\beta \prod_{k=1}^n x_k^{\frac{\beta}{2}(a+1)-1} (1-x_k)^{\frac{\beta}{2}(b+1)-1},
\end{align}
and the corresponding distribution is called 
 the \emph{real Jacobi $\beta$-ensemble}.
The real Jacobi $\beta$-ensemble also arises from the study of multivariate analysis of variance (MANOVA) in statistics: if $\beta=2$ and $a,b\in\Z_{\geq 0}$, then  \eqref{eq:PDF_Jacobi} is the joint eigenvalue density of an $n\times n$ MANOVA model with parameters $n_1=n+a$ and $n_2=n+b$.

We study point process limits of the CJ$\beta$E and RO$\beta$E, together with the scaling limits of some related objects, in particular the limits of the normalized characteristic polynomials. Our approach follows the framework introduced in \cite{BVBV_op} and \cite{BVBV_szeta}. This framework, together with a high level description of our main results is  summarized in the following outline.

\begin{enumerate}
    \item \textbf{Differential operators from probability measures.} \cite{BVBV_op} describes how the spectral information (the modified Verblunsky coefficients) of a finitely supported probability measure on the unit circle can be used to construct a differential operator (a Dirac operator) with a pure point real spectrum. The spectrum of the constructed differential operator is the periodic lifting of the angles corresponding to the support of the probability measure, see Proposition \ref{prop:unitary_repr} for the precise statement. We summarize the  background and the relevant results in Section \ref{s:DirOp}.
    
    \item \textbf{Random Dirac operators.} \cite{BNR2009} and \cite{KillipNenciu} provide constructions for random probability measures on the unit circle where the support of the measure is given by the CJ$\beta$E and RO$\beta$E, respectively,  and the distribution of the modified Verblunsky coefficients can be explicitly described, see Theorems \ref{thm:CJVer} and \ref{thm:OrthVer}. 
%    We review these  results in Section \ref{subs:finiteOp}. 
    These constructions together with Proposition \ref{prop:unitary_repr} lead to the construction of the  random differential operators $\CJop$ and $\ROop$ with pure point spectrum. The spectrum of $\CJop$ is distributed as $n \Lambda_n+2\pi n \ZZ$ with $\Lambda_n\sim \textup{CJ}_{n,\beta,\delta}$, and the spectrum of $\ROop$ is distributed as $2n \Lambda_{2n}+4\pi n \ZZ$ where $\Lambda_{2n}\sim \textup{RO}_{2n,\beta, a, b}$, see Section \ref{subs:finiteOp}. The inverses of these differential operators (after a change of basis) are denoted by $\res \CJop$ and $\res \ROop$, these are random  Hilbert-Schmidt integral operators acting on $L^2$ functions of the form $[0,1)\to \R^2$.

    \item  \textbf{Operator level convergence.} The operators $\CJop$ and $\ROop$ and their inverses can be parameterized  in terms of certain random walks in the hyperbolic plane. Under the appropriate scaling these random walks converge to diffusions in the hyperbolic plane. As shown in \cite{BVBV_op}, one can construct random differential operators in terms of these diffusions, these will be called $\Huaop$  and $\Bessop$, respectively. (See Section \ref{subs:limitOp}.) Both of these random differential operators have pure point spectra, the distribution of the point processes of eigenvalues are denoted by $\HPb$ and $\BB_{\beta,a}$, respectively. The process $\HPb$ for $\delta=0$ is  the process $\Sineb$ introduced in \cite{BVBV} as the bulk scaling limit of the Gaussian $\beta$-ensemble.
    The process $\BB_{\beta,a}$ is just a symmetrized (and scaled) version of the square root of the hard edge process $\Bessel_{\beta,a}$ introduced in \cite{RR}.

    We will prove that in  appropriate couplings we have the operator level convergence
    \begin{align}\label{eq:CJ_lim}
    \|\res \CJop  -\res \Huaop\|_{HS} \to 0 \text{ almost surely as $n\to \infty$,} \\
    \|\res \ROop  -\res \Bessop\|_{HS} \to 0 \text{ almost surely as $n\to \infty$.} \label{eq:RO_lim}
    \end{align}
    
%In the last statement $b=b_n$ can depend on $n$ XXX. 
The precise version of  these results are stated in Theorems \ref{thm:CJ_op_limit} and \ref{thm:RO_op_limit} in Section \ref{subs:Limits}. These results identify the point process scaling limits of the ensembles CJ$\beta$E and RO$\beta$E as the 
point processes $\HPb$ and $\BB_{\beta,a}$. (See Corollaries \ref{cor:CJ_lim} and \ref{cor:RO_lim}.) The distribution of the point process $\HPb$ can be characterized via its counting function using coupled systems of stochastic differential equations. Two equivalent characterizations are given in   Theorems \ref{thm:KSsde} and \ref{thm:HP_2piZ} in Section \ref{subs:limitproc}. 

\item \textbf{Convergence of characteristic polynomials.} \cite{BVBV_szeta} introduced the secular function for  a Dirac operator $\tau$ which is an entire function with zero set given by the spectrum of $\tau$. This is a generalization of the normalized characteristic polynomial of a unitary matrix. We review the definition in Section \ref{s:DirOp}.
\cite{BVBV_szeta} also showed that results of the form of \eqref{eq:CJ_lim} and \eqref{eq:RO_lim} (together with similar statements on the so-called integral trace) imply that the scaled and normalized characteristic polynomials of CJ$\beta$E and RO$\beta$E converge to the secular functions of the  operators $\Huaop$ and $\Bessop$. These results are stated as part of  Corollaries \ref{cor:CJ_lim} and \ref{cor:RO_lim}. Theorems \ref{thm:zetaHP} and \ref{thm:zetaB} provide two separate characterizations of the limiting random entire functions: by describing the joint distribution of the Taylor coefficients, and a characterization using entire function valued stochastic differential equations. 
\end{enumerate}

%\subsubsection*{Historical notes} 
For the circular Jacobi $\beta$-ensemble the operators $\CJop$ and $\Huaop$ were introduced in \cite{BVBV_op}, and the convergence \eqref{eq:CJ_lim} was stated as a conjecture. (More precisely: as a statement to be proved in a future paper.) 
 In \cite{AN2019} Assiotis and Najnudel showed the existence of the  point process limit of  the  circular Jacobi $\beta$-ensemble by providing a coupling of the scaled  finite ensembles so that they posses an a.s.~point process limit. However their result does not provide an explicit characterization for the limiting point process. 
 
 Our main new contributions for the study of the scaling limits of CJ$\beta$E are the 
operator level convergence of Theorem \ref{thm:CJ_op_limit}, the various  characterizations of the limit point process $\HPb$ (Theorems \ref{thm:KSsde} and \ref{thm:HP_2piZ}), 
  and the description and characterization of the limit of the normalized characteristic polynomials (Corollary \ref{cor:CJ_lim} and Theorem \ref{thm:zetaHP}). We also state results on the large gap asymptotics of the  point process $\HPb$, a central limit theorem on the counting function of $\HPb$, and a process level transition from $\HPb$ to the $\Sineb$ process (see Theorems \ref{thm:largegap} and \ref{thm:CLT_transition}). 
 Some of our results are extensions and generalizations of corresponding results for the circular $\beta$-ensemble and the $\Sineb$ process proved in \cite{KS}, \cite{BVBV_op}, \cite{BVBV_19}, \cite{BVBV_szeta}. 

In the $\beta=1, 2, 4$ cases the limiting point processes have been described via their $n$-point correlation functions in \cite{FN_2001}. In \cite{Liu2017} the limiting correlation functions were derived in the case when $\beta$ is an even integer, together with exact formulas for expectations of products of characteristic polynomials. (Note that the normalization for the characteristic polynomials in \cite{Liu2017} is different from ours.) \cite{FLT_2021} provides corrections to these results in the case when $\beta$ is an even integer or equal to 1. 
Scaling limits of characteristic polynomials of classical random matrix ensembles were also studied in \cite{CNN} and  \cite{chhaibi2019limiting}.

A version of the first three steps of the  outline above was carried out by Holcomb and Moreno Flores in \cite{HM2012} for the real Jacobi $\beta$-ensemble. Using the change of variables of \eqref{eq:changeofvarJ}, their results also imply the point process level convergence of 
RO$\beta$E. The proof in \cite{HM2012} relies on a tridiagonal representation of the real Jacobi $\beta$-ensemble together with the operator convergence approach for studying the hard edge limit, introduced in \cite{RR} for the Laguerre $\beta$-ensemble. 
\cite{BVBV_op} provided a representation of  the hard edge limit operator as a random Dirac operator.
\cite{H2018} provides various descriptions and properties of the limiting (hard edge) point process. 
Our main new contributions for the study of RO$\beta$E   are the existence and description of the limit of the normalized characteristic polynomials (Corollary \ref{cor:RO_lim} and Theorem \ref{thm:zetaB}), and  a new approach to prove the point process limit via operator convergence (Theorem \ref{thm:RO_op_limit}). 

\subsubsection*{Outline of the paper} 

In Section \ref{s:DirOp} we outline the used operator theoretic framework, the presentation will mostly follow that of \cite{BVBV_op} and \cite{BVBV_szeta}. In Section \ref{s:RandomDir} we introduce the random differential operators corresponding to the finite ensembles and their limits. Section \ref{s:results} states our precise results, including the  description of the limiting point processes and random analytic functions. 
Sections \ref{s:tools}, \ref{s:pathconv}, and \ref{s:oplim} provide the proofs for the operator convergence results, while Section \ref{s:limitingobjects} contains the proofs of the statements of the properties and characterizations of the limiting objects.

\section{The operator theoretic framework} \label{s:DirOp}
% \yun{add random analytic functions here?}

This section collects all the deterministic operator theoretic ingredients. We describe the type of differential and integral operators we consider, the definition of the secular function, and how these objects can be used to study finitely supported probability measures on the unit circle.

\subsection{Dirac operators}\label{subs:Diracop}

We start by collecting some basic facts about differential operators of the form
\begin{equation}\label{def:Dirop}
	\tau: f \to R^{-1}(t)Jf', \qquad f:[0,1)\to \R^2, \quad J=\mat{0}{-1}{1}{0},
\end{equation}
where $R(t)$ is a  positive definite real symmetric $2\times2$ matrix valued function on $[0,1)$. These differential operators are called Dirac operators, for more details see \cite{Weidmann} or \cite{BVBV_op}.

We consider differential operators of the form \eqref{def:Dirop} where the matrix valued function $R(t)$ is defined from a locally bounded measurable function $x+i y:[0,1)\to \HH=\{z\in \CC: \Im z>0\}$ as follows: 
\begin{align}\label{eq:R}
 R=\frac12 X^tX,\quad X=\frac{1}{\sqrt{y}}
\mat{1}{-x}{0}{y}. 
\end{align}
We call $R$ the \emph{weight function}, and $x+i y$ the \emph{generating path} of $\tau$.

The boundary conditions for $\tau$ at 0 and 1 are given by nonzero, non-parallel $\R^2$ vectors $\uu_0, \uu_1$. We will assume that these vectors are normalized so that they satisfy the condition
\begin{align}\label{eq:u_assumption}
\uu_0^t J \uu_1=1.
\end{align}
We will also have the following integrability assumption for the boundary conditions:
\begin{assumption}\label{assumption:1}
    \begin{align}\label{eq:int_assumption}
\int_0^1  \|R(s) \uu_1\| ds<\infty \quad \textup{and} \quad     \int_0^{1} \int_0^t \mathfrak  u_0  ^t R(s) \mathfrak  u_0   \, \mathfrak  u_1^t R(t) \mathfrak  u_1 ds dt<\infty.
\end{align}
\end{assumption}
Under these conditions $\tau$ will be self-adjoint on the following domain: 
\begin{equation}\label{Dir:domain}
\mathrm{dom}(\tau)=\{v\in L^2_R\cap\text{AC}:  \tau v\in L^2_R, \lim_{s\to 0}v(s)^tJ\uu_0=0,\lim_{s\to 1} v(s)^tJ\uu_1=0\}.
\end{equation}
Here $L^2_R$ is the $L^2$ space of functions $f:[0,1)\to\R^2$ with the $L^2$ norm $||f||_R^2=\int_0^1 f^tRfds$, and $\text{AC}([0,1))$ is the set of absolutely continuous real functions on $[0,1)$. We will use the notations $\Dirop(R,\uu_0,\uu_1)$ or $\Dirop(x+i y, \uu_0, \uu_1)$ for the the operator $\tau$ defined via \eqref{def:Dirop} and \eqref{eq:R} with boundary conditions $\uu_0, \uu_1$ on the domain \eqref{Dir:domain}. We sometimes replace the $\R^2$ vector by the element in $\RR\cup \{\infty\}$ corresponding to the ratio of its two coordinates: $[a,b]^t$ corresponds to $a/b$ if $b\neq 0$ and $\infty$ if $b=0$.

The inverse of $\tau=\Dirop(x+i y, \uu_0, \uu_1)$ is a Hilbert-Schmidt integral operator on $L^2_R$ with kernel given by 
\begin{equation}\label{eq:HS_kernel}
	K_{\tau^{-1}}(s,t) = \big( \uu_0\uu_1^t1_{s<t}+\uu_1\uu_0^t1_{s\geq t}\big) R(t).
\end{equation}
This means that if $f\in \mathrm{dom}(\tau)$ and $g=\tau f$ then $f(s)=\int_0^1 K_{\tau^{-1}}(s,t) g(t) dt$. The fact that the integral operator is Hilbert-Schmidt follows from the second inequality of  \eqref{eq:int_assumption}, and implies that $\tau$ has a discrete pure point  spectrum with nonzero real eigenvalues $\lambda_k, k\in \ZZ$ that satisfy $\sum_k \lambda_k^{-2}<\infty$. We label the eigenvalues so that they are in an increasing order with  $\lambda_{-1}<0<\lambda_0$.

After the change of variables $\hat \tau=X\tau X^{-1}$,  the inverse $\ttr\tau:=\hat \tau^{-1}$ is an integral operator on the $L^2$ space of functions $f:[0,1)\to\R^2$   with norm $\|f\|^2=\int_0^1 f^t f ds$, and its kernel is given by
\begin{equation}\label{Dir:inverse}
K_{\ttr \tau}(s,t) = \tfrac12\big( a(s)c(t)^t1_{s<t}+c(s)a(t)^t1_{s\geq t}\big), \quad a(s)=X(s)\uu_0,\quad c(s)=X(s)\uu_1.
\end{equation}
We define the \emph{integral trace} of $\ttr \tau$ as 
 the  integral of the trace of the kernel $K_{\ttr \tau}$, and denote it by $\mathfrak{t}_\tau$:
  \begin{align}
\mathfrak{t}_\tau=\int_0^{1} \tr K_{\ttr \tau}(s,s) ds=\tfrac12 \int_0^{1} a(s)^t c(s) ds=\int_0^1 \uu_0^t R(s) \uu_1 ds.
    \label{eq:int_tr}
\end{align}  
By the first inequality of \eqref{eq:int_assumption} the integral trace is finite.

We define the \emph{secular function} of $\tau$ with the expression 
\begin{align}
    \zeta_\tau(z)=e^{-z{\mathfrak t}_\tau}{\det}_2(I-z \,\ttr \tau) =e^{-\tfrac{z}{2} \int_0^{1} a(s)^t c(s) ds} \prod_k (1-z/ \,\lambda_k)e^{z/ \lambda_k}.\label{eq:zeta}
\end{align}
Here ${\det}_2$ is the second regularized determinant, see \cite{Simon_det}. The secular function $\zeta_\tau$ is an entire function with zero set given by $\lambda_k, k\in \ZZ$, it is
an analogue of the normalized characteristic polynomial of a square matrix. (See \cite{BVBV_szeta} for details.)

The next statement provides comparisons for the spectra and secular functions of two Dirac operators.

\begin{proposition}\label{prop:comp}
Let $\tau_1 $, $\tau_2$ be two Dirac operators on $[0,1)$ satisfying assumptions \eqref{eq:u_assumption} and \eqref{eq:int_assumption}.  Denote by $\lambda_{k,i}, \zeta_i, \res_i, \mathfrak t_i$ the eigenvalues,  secular function, resolvent and integral trace of $\tau_i$. 
Let  $\|\cdot \|$ denote the Hilbert-Schmidt norm.
Then 
\begin{align}\label{eq:ev_comp}
\sum_k \left|\lambda_{k,1}^{-1}-\lambda_{k,2}^{-1}\right|^2\le& \|\res_1-\res_2\|^2,
\end{align}
and there is a universal constant $a>1$ so that for all $z\in \CC$
\begin{align}\label{eq:zeta_comp}
|\zeta_1(z)-\zeta_2(z)|\le& \Big(e^{|z| |\mathfrak t_1-\mathfrak t_2|}-1+|z|\big\|\res_1-\res_2\big\|\Big)
a^{|z|^2(\|\res_1\|^2+\|\res_2\|^2)+|z|(|\mathfrak{t}_1|+|\mathfrak{t}_2|)+1}
\end{align}
\end{proposition}
The inequality \eqref{eq:ev_comp} is just the Hoffman-Wielandt inequality in infinite dimensions (see e.g.~\cite{BhatiaElsner}), the bound \eqref{eq:zeta_comp}  follows from standard properties of the regularized determinant \cite{Simon_det} (see Proposition 21 in \cite{BVBV_szeta} for additional details). Proposition \ref{prop:comp} shows that the Hilbert-Schmidt convergence of Dirac operators implies the convergence of the spectra, and if the integral traces converge as well then we have uniform on compacts convergence of the secular functions. 

The end points of a Dirac operator can be classified as limit circle or limit point based on the integrability of the solutions of $(\tau-\lambda)u=0$ near that end point. By the Weyl's alternative theorem (e.g. Theorem 5.6 in \cite{Weidmann}) the integrability of the solutions does not depend on $\lambda$. Hence one can choose $\lambda=0$, and just check the integrability of the constant vectors. 
Since  $R(t)$ is locally bounded near 0, the  left endpoint of the interval $[0,1)$ is \emph{limit circle} with respect to the weight function $R$: for any $v\in \RR^2$ the function $v^t Rv$ is integrable near $0$. Assumption \eqref{eq:int_assumption} shows that $v R v$ is integrable near 1 if $v\parallel \uu_1$, but that might not be the case if $v \! \not {\!\parallel} \, \uu_1$. This shows that the right endpoint could be limit circle or \emph{limit point}.

For certain applications of the limiting objects, it is more convenient to consider operators that have 0 as the endpoint that could possibly be limit point. In this case the domain of the operator is $(0,1]$, and we have to modify are setup and assumptions. This \emph{reversed framework} will be introduced in Section \ref{subs:reverse}, where we also discuss 
other transformations of Dirac operators.

\subsection{Dirac operators for finitely supported probability measures on the unit circle}

We review the construction given Section 3 of \cite{BVBV_szeta}  that shows how a finitely supported probability measure on the unit circle can be represented using a Dirac operator of the form \eqref{def:Dirop}. (See also Section 5 of \cite{BVBV_op}.)

Let $\mu$ be a probability measure whose support is a set of $n$ distinct points $e^{i\lambda_j}, 1\le j \le n$ on the unit circle, and assume $\mu(\{1\})=0$. The characteristic polynomial of $\mu$, normalized at $1$,  is  defined as
\begin{align}\label{eq:charpol}
p_\mu(z)=\prod_{j=1}^n\frac{z-e^{i\lambda j}}{1-e^{i\lambda j}}.
\end{align}
For $0\le k\le n$, the $k$th orthogonal polynomial $\Phi_k(z)$ is defined as the unique polynomial with main coefficient $1$ of degree $k$ that is orthogonal to $1,\ldots,z^{k-1}$ in $L^2(\mu)$. 
We denote by $\Psi_k(z)=\frac{\Phi_k(z)}{\Phi_k(1)}$ the normalized orthogonal polynomials. Note that we have $\Phi_0=\Psi_0=1$ and  $p_\mu=\Psi_n$.
For $0\le k\le n$ we define $\Phi_k^*, \Psi_k^*$ as the
reversed polynomials 
\[
\Phi^*_k(z)=z^k \overline{\Phi_k(1/\bar z)}, \qquad \Psi^*_k(z)=z^k \overline{\Psi_k(1/\bar z)}.
\]
The vector $\binom{\Phi_k}{\Phi_k^*}$ satisfies the  Szeg\H{o} recursion \cite{OPUC1}:
\begin{align}\label{eq:Szego}
    \binom{\Phi_{k+1}(z)}{\Phi_{k+1}^*(z)}=A_k \mat{z}{0}{0}{1}\binom{\Phi_{k}(z)}{\Phi_{k}^*(z)}, \qquad 0\le k\le n-1.
\end{align}
Here $A_k=\mat{1}{-\bar \alpha_k}{-\alpha_k}{1}$, the complex numbers $\alpha_0, \dots, \alpha_{n-1}$ are called the Verblunsky coefficients. They satisfy $|\alpha_k|<1$ for $0\le k\le n-1$ and $|\alpha_{n-1}|=1$. The normalized orthogonal polynomials $\Psi_k, \Psi_k^*$ satisfy a similar recursion as \eqref{eq:Szego}, with the matrix 
\[
\tl A_k=\mat{\frac{1}{1-\gamma_k}}{-\frac{\gamma_k}{1-\gamma_k}}{-\frac{\bar \gamma_k}{1-\bar \gamma_k}}{\frac{1}{1-\bar \gamma_k}}
\]
in place of $A_k$. The complex numbers $\gamma_k$, $0\le k\le n-1$ are called the modified or deformed Verblunsky coefficients (see \cite{BNR2009}). They satisfy 
\begin{align}\label{eq:modifiedV}
\gamma_k=\bar \alpha_k \prod_{j=0}^{k-1} \frac{1-\bar \gamma_j}{1-\gamma_j}, \qquad 0\le k\le n-1,
\end{align}
from which it follows that $|\gamma_k|=|\alpha_k|$.

% The spectral information of $\mu$ is encoded in the sequence of  modified Verblunsky coefficients via the Szeg\H{o} recursion.   can be used to produce a piece-wise constant path 

Define $w_k,v_k\in \mathbb R$  with
\begin{align}\label{def:wv}
   \frac{2\gamma_k}{1-\gamma_k}= w_k-iv_k. 
\end{align}
Set $x_0=0$, $y_0=1$, and define recursively
\begin{align}\label{xyrec}
x_{k+1}=x_k+v_k y_k, \qquad y_{k+1}=y_k(1+w_k), \qquad 0\le k\le n-1.
\end{align}
Note that  $|\gamma|\le 1$ implies $\Re \frac{2\gamma}{1-\gamma}\ge -1$, and we have equality  if and only if $|\gamma|=1$, $\gamma\neq 1$.
Hence $y_k>0$ for $1\le k\le n-1$ and $y_n=0$. The following proposition was proved in \cite{BVBV_szeta}. 

\begin{proposition}[\cite{BVBV_szeta}]
\label{prop:unitary_repr}
Set $x(t)+i y(t)=x_{\lfloor nt \rfloor}+i y_{\lfloor nt \rfloor}$ for $t\in [0,1]$.
Let
\begin{align}\label{eq:dscrtDirop}
\tau = R^{-1} \mat{0}{-1}{1}{0} \frac{d}{dt}, \qquad R=\frac{X^t X}{2\det X}, \qquad X=\mat{1}{-x}{0}{y},
\end{align}
with boundary conditions $\mathfrak  u_0  =[1,0]^t$, $\mathfrak  u_1 =[-x(1),-1]^t$.

Then $\tau$ satisfies our assumptions, the spectrum of $\tau$ is given by the set 
\[
\spec \tau=\{n \lambda_k+2\pi n j: 1\le k\le n, j\in \ZZ\},
\]
and the secular function of $\tau$ satisfies
\begin{align}\label{eq:char_pol}
\zeta_\tau(z)=
p_\mu(e^{iz/n})e^{-iz/2}=
 \prod_{j=1}^n \frac{\sin(\lambda_j/2-z/(2n))}{\sin(\lambda_j/2)}.
\end{align}
\end{proposition}

\section{Random Dirac operators}\label{s:RandomDir}

This section introduces the random Dirac operators corresponding to the finite ensembles and to their limits.

\subsection{Operators for the finite ensembles} \label{subs:finiteOp}

The results of this section provide descriptions of random probability measures with support given by the CJ$\beta$E and RO$\beta$E, respectively, where the joint distribution of the modified Verblunsky coefficients can be described explicitly. 

% We first introduce a couple of distributions that will play important roles.

\begin{definition}\label{def:Theta}
For  $a>0$ and $\Re \delta>-1/2$ we denote by $\Theta(a+1,\delta)$ the distribution on $\{|z|<1\}$ that has probability density function
\begin{align}
	 c_{a,\delta} (1-|z|^2)^{a/2-1}(1-z)^{\bar \delta} (1-\bar z)^{\delta},
	\end{align}
where $c_{a,\delta}=\frac{\Gamma(a/2+1+\delta)\Gamma(a/2+1+\bar \delta)}{\pi \Gamma(a/2)\Gamma(a/2+1+\delta+\bar \delta)}$.

We extend the definition for the $a=0$, $\Re \delta>-1/2$ case as follows:  $\Theta(1,\delta)$ is the distribution on $\{|z|=1\}$ with probability density function
\begin{align}
\tfrac{\Gamma(1+\delta)\Gamma(1+\bar \delta)}{\Gamma(1+\delta+\bar \delta)} (1-z)^{\bar \delta} (1-\bar z)^{\delta}.
\end{align}

\end{definition}

\begin{definition}\label{def:Beta*}
For $s,t>0$ let $\tl{\mathrm{B}}(s,t)$ denote the scaled (and flipped) beta distribution on $(-1,1)$ that has probability density function 
	\begin{align*}
\tfrac{2^{1-s-t}\Gamma(s+t)}{\Gamma(s)\Gamma(t)}(1-x)^{s-1}(1+x)^{t-1}.
	\end{align*} 
\end{definition}

\begin{theorem} [Theorems 3.2 and 3.3 of \cite{BNR2009}] \label{thm:CJVer} 
For given $\beta>0$, $\Re \delta>-1/2$ 
and $n\ge 1$ let $\mu=\mu_{n,\beta, \delta}^{\textup{cj}}$ be the random probability measure $\mu=\sum_{k=1}^n r_k \delta_{e^{i \theta_k}}$ on the unit circle where $(\theta_1,\dots, \theta_n)$ and $(r_1, \dots, r_n)$ are independent, the joint density of $\theta_k, 1\le k \le n$ is given by \eqref{eq:PDF_cjacobi}, and the joint density of $r_k, 1\le k \le n-1$ is given by $\frac{1}{C_{n,\beta}} \prod_{k=1}^n r_k^{\beta/2-1}$. In other words, $\mu$ is a probability measure where the support has distribution CJ$\beta$E, and the weights are Dirichlet$(\beta/2, \dots, \beta/2)$ distributed, independently of the support.

Then the modified Verblunsky coefficients $\gamma_0, \dots, \gamma_{n-1}$ of $\mu$ are independent, and $\gamma_k$ has distribution $\Theta(\beta(n-k-1)+1,\delta)$ for $0\le k\le n-1$.
% For $0\leq k\leq n-2$, the density of $\gamma_{k}$  with respect to the Lebesgue measure on the unit disk is given by
% 	\begin{align}
% 	p_{n,k,\delta}(z) =  c_{n,k,\delta} (1-|z|^2)^{\frac{\beta}{2}(n-k-1)-1}(1-z)^{\bar \delta} (1-\bar z)^{\delta},
% 	\end{align}
% 	where
% 	\[
% 	c_{n,k,\delta}=\frac{\Gamma(\frac{\beta}{2} (n-k-1)+1+\delta)\Gamma(\frac{\beta}{2} (n-k-1)+1+\bar \delta)}{\pi \Gamma(\frac{\beta}{2} (n-k-1))\Gamma(\frac{\beta}{2} (n-k-1)+1+\delta+\bar \delta)}.
% 	\]
% 	The density of $\gamma_{n-1}$ with respect to the uniform measure on the unit circle is
% 	\begin{align}\label{eq:boundaryMV_pdf}
% 	p_{n,n-1,\delta}(z) = \frac{\Gamma(1+\delta)\Gamma(1+\bar \delta)}{\Gamma(1+\delta+\bar \delta)} (1-z)^{\bar \delta} (1-\bar z)^{\delta}.
% 	\end{align}
\end{theorem}

\begin{theorem}[Theorem 2 of \cite{KillipNenciu}, Proposition 4.5 in \cite{KK}] \label{thm:OrthVer}
For given $\beta>0$, $a, b>-1$ 
and $n\ge 1$ let $\mu=\mu_{2n,\beta, a, b}^{\textup{o}}$ be the random probability measure $\mu=\sum_{k=1}^n \frac12 r_k (\delta_{e^{i \theta_k}}+\delta_{e^{-i \theta_k}})$ on the unit circle where $(\theta_1,\dots, \theta_n)$ and $(r_1, \dots, r_n)$ are independent, the joint density of $\theta_k, 1\le k \le n$ is given by \eqref{eq:PDF_ortho}, and the joint density of $r_k, 1\le k \le n-1$ is given by $\frac{1}{C_{n,\beta}} \prod_{k=1}^n r_k^{\beta/2-1}$. 

Then the Verblunsky coefficients $\alpha_0, \dots, \alpha_{2n-1}$ corresponding to $\mu$ are real,  independent of each other. We have $\alpha_{2n-1}=-1$, and the distribution of $\alpha_k,0\le k\le 2n-2$ is given by  
\begin{align*}
    	  \alpha_{k}\sim \begin{cases}
	    \mathrm{\tl B}\big(
	\tfrac{\beta}{4}(2n-k+2a)
, \tfrac{\beta}{4}(2n-k+2b)
\big),\quad &\text{if $k$ is even,}\\
\mathrm{\tl B}\big(\tfrac{\beta}{4}(2n-k+2a+2b+1),
	\frac{\beta}{4}(2n-k-1)
	\big),\quad &\text{if $k$ is odd.}
	    \end{cases}
\end{align*}
Since all the Verblunsky coefficients are real, we have $\gamma_k=\alpha_k$ for all $0\le k\le 2n-1$.
\end{theorem}

Theorems \ref{thm:CJVer} and \ref{thm:OrthVer} together with Proposition \ref{prop:unitary_repr} provide random Dirac operator representations for the CJ$\beta$E and RO$\beta$E. 

\begin{definition}
We denote by  $\CJop$ the random Dirac operator constructed from the random probability measure $\mu_{n,\beta, \delta}^{\textup{cj}}$ of Theorem \ref{thm:CJVer} using Proposition \ref{prop:unitary_repr}.
We denote by  $\ROop$ the random Dirac operator constructed from the random probability measure $\mu_{2n,\beta, a, b}^{\textup{o}}$ of Theorem \ref{thm:OrthVer} using Proposition \ref{prop:unitary_repr} .
\end{definition}

% Note that for the RO$\beta$E   equation \eqref{eq:modifiedV} implies that  $\gamma_k=\alpha_k$. 

The modified Verblunsky coefficients are independent for both $\mu_{n,\beta, \delta}^{\textup{cj}}$ and $\mu_{2n,\beta, a, b}^{\textup{o}}$. Hence the sequence $x_k+i y_k$ defined by the recursion \eqref{xyrec} is a Markov chain for both of these random measures. The generating paths of the $\CJop$ and $\ROop$ operators are just these Markov chains embedded into continuous time.

% Since for both the CJ$\beta$E and RO$\beta$E the modified Verblunsky coefficients are independent, 
% the recursion \eqref{xyrec} used to define the piece-wise constant driving path of $\CJop$ and $\ROop$ defines a Markov chain in both cases.  Note that in the RO$\beta$E case the variables $\gamma_k, 0\le k\le 2n-1$ are real, and hence $v_k=x_k=0$ for all $k$. 

\subsection{The limiting operators}
\label{subs:limitOp}

As we show below, the generating paths of both $\CJop$ and $\ROop$ approximate certain diffusions in $\HH$, and the operators themselves approximate the Dirac operators built from these diffusions. In this section we introduce the two limiting operators. 

For the rest of the paper, we set 
\begin{align}
    \upsilon_\beta(t)=-\tfrac{4}{\beta}\log(1-t). \label{eq:timechange}
\end{align} 

\subsubsection*{Hua-Pickrell operator}

Fix $\beta>0$ and $\delta\in\CC$ with $\Re \delta>-1/2$. Let $B_1, B_2$ be independent standard Brownian motion, and let $x_t+i y_t, t\ge 0$ be the strong solution of the SDE
\begin{align}\label{eq:HPsde}
dy=\left( -\Re \delta dt +  dB_1 \right)y,\quad dx=\left(\Im \delta dt +  dB_2 \right)y, \quad y(0)=1, x(0)=0.
\end{align}

\begin{proposition}[Proposition 31 of \cite{BVBV_op}] \label{prop:hpoperator}
	Let $x(t)+i y(t)$ be defined via \eqref{eq:HPsde}. The limit $q=\lim_{t\to \infty} x(t)$  exists, and it is non-zero with probability one.
	Define $\tl x(t)=x(\upsilon_\beta(t))$,$\tl y(t)=y(\upsilon_\beta(t))$, and set $\uu_0=[1,0]^t$, $\uu_1=[-q,-1]^t$. 
	Then the random Dirac operator 
	$\Huaop=\Dirop(\tl x+i \tl y, \uu_0, \uu_1)$ satisfies the assumptions of Section \ref{subs:Diracop}. 
	%In particular, it is almost surely self-adjoint on the appropriate domain,  its inverse is Hilbert-Schmidt, and its secular function is well-defined.
\end{proposition}

We record the following estimates for $\tl x, \tl y$ from the proof of Proposition 31 of \cite{BVBV_op}. 
% Note that the SDE \eqref{eq:HPsde} can be solved explicitly and the solution is given by
% 	\begin{equation}\label{eq:hpsol}
% 	y(t)=e^{B_1(t)-(\Re\delta+\frac12)t},\quad x(t)=\int_0^t y(s)dB_2(s)+\Im\delta\int_0^t y(s) ds.
% 	\end{equation}
% Since $\Re\delta+\frac12>0$, $y(t)$ converges to $0$ a.s. as $t\to\infty$ by the sub-linearity of the Brownian motion, and the limit $q = \int_0^\infty y(s)dB_2(s)+\Im\delta\int_0^\infty y(s) ds$ is a.s. finite. By the independence of $B_1$ and $B_2$, and the law of the iterated logarithm, 
For any $\eps>0$ small there exists a random finite $C=C(\eps)$ such that
\begin{equation}\label{ineq:hpsde}
	C^{-1}(1-t)^{\frac{4}{\beta}(\Re\delta+\frac12+\eps)}\leq \tl y(t)\leq C(1-t)^{\frac{4}{\beta}(\Re\delta+\frac12-\eps)},\quad |q-\tl x(t)|\leq C(1-t)^{\frac{4}{\beta}(\Re\delta+\frac12-\eps)}.
\end{equation}
% This also implies that 
% \[
% q=\int_0^\infty y(s)dB_2(s)+\Im\delta\int_0^\infty y(s) ds
% \]
% is a.s.~finite. 
The distribution of $q=\lim_{t\to\infty} x(t)$  was identified in \cite{BCFY2001}. 
\begin{definition}\label{def:Pearson}
For $m>1/2$ and $\mu\in \R$ we denote by $P_{IV}(m,\mu)$ the distribution of the (unscaled) Pearson type IV distribution on $\R$ that has density function 
\begin{align}
    \frac{2^{2m-2}|\Gamma(m+\frac{\mu}{2} i)|^2}{\pi \, \Gamma(2m-1)} (1+x^2)^{-m}e^{-\mu \arctan x}.
\end{align}
\end{definition}

\begin{theorem}[\cite{BCFY2001}]\label{thm:q}
 The random variable $q$ in Proposition \ref{prop:hpoperator} has $P_{IV}(\Re \delta+1,-2 \Im \delta)$ distribution.
%  the following density function on $\R$:
%  \begin{align}\label{eq:HP_q_pdf}
%      f_\delta(x)=\frac{4^{\Re\delta}|\Gamma(1+\delta)|^2}{\pi\Gamma(1+2\Re\delta)} e^{2\Im \delta \arctan x} (1+x^2)^{-\Re \delta-1}.
%  \end{align}
\end{theorem}
There is an interesting connection between the distributions $P_{IV}$ and $\Theta$: the map $z(e^{i \theta})= - \cot(\theta/2)$ transforms  $\Theta(1,\delta)$ into $P_{IV}(\Re \delta+1,-2 \Im \delta)$. The map  $z$ can be extended to the conformal map $w\to i\frac{w+1}{-w+1}$  from $\{|w|\le 1\}$ to $\{\Im z>0\}$, which provides an isometry between the unit disk and half-plane representations of the hyperbolic plane.  
In other words,  $\Theta(1,\delta)$ and $P_{IV}(\Re \delta+1,-2 \Im \delta)$ are  different representations of the same distribution on the boundary of the hyperbolic plane.

\subsubsection*{Hard edge operator}

The point process scaling limit of the Laguerre $\beta$-ensemble near the hard edge was identified by Ram\'{\i}rez and Rider in \cite{RR} as the  spectrum of the following random Sturm-Liouville differential operator: 
\begin{align}\label{eq:hardop}
\mathfrak{G}_{\beta,a} f(x)= -e^{(a+1)x+\frac{2}{\sqrt{\beta}}W(x)} \partial_x \left(e^{-ax-\frac{2}{\sqrt{\beta}}W(x)} \partial_x \, f(x) \right).
\end{align}
Here $W(x)$ is a standard Brownian motion, and the operator acts on functions $[0,\infty)\to \R$ with Dirichlet boundary condition at 0  and Neumann boundary condition at $\infty$. 

 \cite{BVBV_op} provided a Dirac operator representation for $\mathfrak{G}_{\beta,a}$, we summarize the result below.

\begin{proposition}[Theorem 30 of \cite{BVBV_op}]\label{prop:hardedge}
Fix $\beta>0,a>-1$, and let $B$ be a standard Brownian motion. We set  $y(t)=e^{-\frac{\beta}{4}(2a+1)t-B(2t)}$, $\tl y(t)=y(\upsilon_\beta(t))$, $\uu_0=[1,0]^t$, and $\uu_1=[0,-1]^t$.

Then the operator $\Bessop:=\Dirop (i \tl y, \uu_0, \uu_1)$ satisfies the assumptions of Section \ref{subs:Diracop}, 
%It is almost surely self-adjoint on the appropriate domain, 
%its inverse is Hilbert-Schmidt,
and its spectrum is symmetric about 0: $\lambda_{-k}=-\lambda_{k-1}, k\ge 1$.

Moreover, the set $\{\tfrac{1}{16} \lambda_0^2, \tfrac{1}{16}\lambda_1^2, \dots\}$ has the same distribution as the spectrum of the hard edge operator $\mathfrak{G}_{\beta,a}$ defined in \eqref{eq:hardop}

% \[
% \mathrm{spec}\,\BB_{\beta,a} = 4\sqrt{\mathrm{spec}\,\mathfrak{G}_{\beta,a}}\cup (-4\sqrt{\mathrm{spec}\,\mathfrak{G}_{\beta,a}}),
% \]
% where $\mathfrak{G}_{\beta,a}$ is the Sturm-Liouville differential operator appearing in the hard edge limit of the Laguerre $\beta$-ensemble  in \cite{RR}. 
\end{proposition}

\begin{remark}
Theorem 30 of \cite{BVBV_op} is 
stated in a slightly different (but equivalent) way.  With the notations of Proposition \ref{prop:hardedge} the statement of that theorem is about the operator $\Besselop=\Dirop(i \tl y^{-1}, \uu_1, \uu_0)$. Note however that conjugating $\Bessop$ with the permutation matrix transposing the first and second coordinate in $\R^2$ gives $-\Besselop$, and since the spectra of $\Bessop$ and $\Besselop$ are symmetric about 0, the statement of the proposition follows. 
\end{remark}

\section{Precise results} \label{s:results}

We are  ready to state our results in a precise form. 

\subsection{Convergence of random operators and normalized characteristic polynomials}
\label{subs:Limits}

\begin{theorem}\label{thm:CJ_op_limit}
Fix $\beta>0$ and $\Re \delta>-1/2$. Then there is a coupling of the random operators $\CJop, n\ge 1$ and $\Huaop$ so that $\|\res \CJop-\res \Huaop\|_{HS}$ and $\intr_{\CJop}-\intr_{\Huaop} $ both converge to 0 almost surely as $n\to \infty$.  
\end{theorem}
From Theorem \ref{thm:CJ_op_limit} and Proposition \ref{prop:comp} we immediately get the following corollary. 

\begin{corollary}\label{cor:CJ_lim}
Consider the coupling of Theorem \ref{thm:CJ_op_limit}. Denote by $\Lambda_n$ the eigenangles of $\CJop$ inside $(-\pi,\pi]$, and let $\lambda_{k,n}, k\in \Z$ be the sequence of ordered elements of the set $n \Lambda_n+2\pi n \Z$ with $\lambda_{-1,n}<0< \lambda_{0,n}$.
Let $p_n(z)$ be the normalized characteristic polynomial of $\Lambda_n$ defined via \eqref{eq:charpol}. Denote by $\HPb=\{\lambda_{k,\ttHP}, k\in \Z\}$  the ordered spectrum of the operator $\Huaop$, and by $\zeta^{\ttHP}_{\beta, \delta}$ the secular function of $\Huaop$. Then 
\begin{align}
    \sum_k |\lambda_{k,n}^{-1}-\lambda_{k,\ttHP}^{-1}|^2\to 0 \qquad &\text{almost surely as $n\to \infty$},\\
|p_n(e^{iz/n})e^{-iz/2}-\zeta^{\ttHP}_{\beta, \delta}(z)|\to 0 \qquad &\text{almost surely, uniformly on compacts as $n\to \infty$}. 
\end{align}
In particular, if $\Lambda_n\sim \textup{CJ}_{n,\beta,\delta}$ then $n\Lambda_n \Rightarrow \HPb$.

\end{corollary}

The following theorem and its corollary state the corresponding result for the real orthogonal ensemble.

\begin{theorem}\label{thm:RO_op_limit}
Fix $\beta>0$ and $a, b>-1$. Then there is a coupling of the random operators $\ROop, n\ge 1$ and $\Bessop$ so that $\|\res \ROop-\res \Bessop\|_{HS}$ converges to 0 almost surely as $n\to \infty$.  
\end{theorem}

Note that since the driving paths are purely imaginary, we have $\intr_{\ROop}=\intr_{\Bessop}=0$.

\begin{corollary}\label{cor:RO_lim}
Consider the coupling of Theorem \ref{thm:RO_op_limit}. Denote by $\Lambda_{2n}$ the eigenangles of $\ROop$ inside $(-\pi,\pi]$, and let $\lambda_{k,2n}, k\in \Z$ be the ordered elements of the set $2n \Lambda_{2n}+4\pi n \Z$ with $\lambda_{-1,2n}<0< \lambda_{0,2n}$.
Let $p_{2n}(z)$ be the normalized characteristic polynomial of $\Lambda_{2n}$ defined via \eqref{eq:charpol}. Denote by $\BB_{\beta,a}=\{\lambda_{k,\ttB}, k\in \Z\}$  the ordered spectrum of the operator $\Bessop$, and by $\zeta^{\ttB}_{\beta,a}$ the secular function of $\Bessop$. Then 
\begin{align}
    \sum_k |\lambda_{k,2n}^{-1}-\lambda_{k,\ttB}^{-1}|^2\to 0 \qquad &\text{almost surely as $n\to \infty$},\\
|p_{2n}(e^{iz/(2n)})e^{-iz/2}-\zeta^{\ttB}_{\beta,a}(z)|\to 0 \qquad &\text{almost surely, uniformly on compacts as $n\to \infty$}. 
\end{align}
Moreover, if $\Lambda_{2n}\sim \textup{RO}_{2n,\beta,a,b}$ then $2n\Lambda_{2n} \Rightarrow \BB_{\beta,a}$.
\end{corollary}

\subsection{Characterization of the limiting point processes}
\label{subs:limitproc}

The point process $\HPb$ is a generalization of the $\Sineb$ process: $\operatorname{HP}_{\beta, 0}=\Sineb$. The $\Sineb$ process has various descriptions via its counting function using  stochastic differential equations, we will show that these descriptions can be extended to the process $\HPb$ as well.

\begin{theorem}\label{thm:KSsde}
 Let $\beta>0$, $\delta\in \CC$ with $\Re \delta>-1/2$. Let $Z=B_1+i B_2$ be a standard complex Brownian motion, and let $\theta\in (-\pi,\pi]$ be a random variable independent of $Z$ so that $e^{i \theta}$ has distribution $\Theta(1,\delta)$.
 
%  \begin{align}\label{eq:ThetaPDF}
%  f_\delta(\theta) =\frac{1}{2\pi}\frac{|\Gamma(1+\delta)|^2}{\Gamma(1+2\Re\delta)}(1-e^{i\theta})^{\bar{\delta}}(1-e^{-i\theta})^{\delta}.
%  %&=\frac{2^{2\Re\delta-1}e^{-\Im\delta\pi}|\Gamma(1+\delta)|^2}{\pi\Gamma(1+2\Re\delta)} \sin(\frac{t}{2})^{2\Re\delta}e^{\Im\delta t}.
%  \end{align}
There is a unique process $\psi_\lambda(t)$ with $t\in (0,1]$, $\lambda\in \R$ that is continuous in both variables, and for each $\lambda\in \R$ the process $t\to \psi_\lambda(t)$ is a strong solution of 
 \begin{align}\label{eq:HP_psi_sde}
     d\psi_\lambda=\lambda dt + \Re[(e^{-i\psi_\lambda}-1)(\tfrac{2}{\sqrt{\beta t}}dZ-i\delta\tfrac{4}{\beta t}dt)], \qquad \lim_{t\to 0} \psi_\lambda(t)=0.
 \end{align}
 The point process $\HPb$ has the same distribution as the random set
 \begin{align}
     \Xi=\{\lambda\in \R: \psi_\lambda(1)\in \theta+2\pi \ZZ\}.
 \end{align}
\end{theorem}
Note that this is an extension of the Killip-Stoiciu characterization of the $\Sineb$ process, see \cite{KS}, \cite{BVBV_op}.
The following theorem provides another, equivalent characterization of $\HPb$, which is an extension of the description of $\Sineb$ given in Proposition 4 of \cite{BVBV}. 

\begin{theorem}\label{thm:HP_2piZ}
Let $\beta>0$, $\delta\in \CC$ with $\Re \delta>-1/2$. Let $Z=B_1+i B_2$ be a standard complex Brownian motion. Then the following SDE system has a unique strong solution on $t\in [0,\infty)$, $\lambda\in \R$  
\begin{align}\label{eq:HP_alpha_sde}
d\alpha_\lambda = \lambda \tfrac{\beta}{4}e^{-\frac{\beta}{4}t}dt +\Re[(e^{-i \alpha_\lambda}-1)(dZ-i\delta dt)],\quad\quad \alpha_{\lambda}(0)=0.
\end{align}
 With probability one the process $\lambda\to \alpha_\lambda(t)$ is increasing for all $t>0$. 
For each $\lambda\in \R$ the limit $\operatorname{sgn}(\lambda)\cdot\lim\limits_{t\to \infty}\frac{1}{2\pi} \alpha_\lambda(t)$ exists almost surely, and it has the same  distribution as the number of points of $\HPb$ in $[0,\lambda]$ for $\lambda\ge 0$ (and in $[\lambda,0]$ for $\lambda<0$). Moreover, if $N(\lambda)$ is the right-continuous version of the function $\lambda\to \lim\limits_{t\to \infty}\frac{1}{2\pi} \alpha_\lambda(t)$, then $N(\cdot)$ has the same distribution as the counting function of the $\HPb$ process. 

\end{theorem}

The diffusion description given in Theorem \ref{thm:HP_2piZ}  allows us to study various properties of  the counting function of the  $\HPb$ process via the SDE \eqref{eq:HP_alpha_sde}. For a given $\lambda\in \R$ the process $\alpha_\lambda$ given by \eqref{eq:HP_alpha_sde} has the same distribution as the unique strong solution of 
	\begin{align}
	d\alpha_{\lambda} = \lambda \tfrac{\beta}{4}e^{-\frac{\beta}{4}t}dt + (\Im\delta(\cos\alpha_{\lambda}-1)-\Re\delta\sin\alpha_{\lambda})dt + 2\sin(\tfrac{\alpha_{\lambda}}{2})dW, \quad \alpha_{\lambda}(0)=0.
	\end{align}
Here $W$ is a standard Brownian motion (which also depends on $\lambda$).

A similar diffusion description for the square root of the hard edge process (spectrum of the operator $\mathfrak{G}_{\beta,a}$ given in \eqref{eq:hardop}) was proved by Holcomb in  \cite{H2018}, building on the results of \cite{RR}. 
% Recall from Proposition \ref{prop:hardedge} that 
% \[
% \BB_{\beta,a}=4\sqrt{\spec \mathfrak{G}_{\beta,a}} \cup (-4\sqrt{\spec\mathfrak{G}_{\beta,a}}).
% \]
Let $M_{a,\beta}(\lambda)$ be the counting function of the $\BB_{\beta,a}$ process and $B$ a standard Brownian motion. 
Then, by Theorem 1.4 of \cite{H2018}, the function $\lambda \to M_{a,\beta}(\lambda)$ has the same distribution as the right continuous version of the function  $\lambda \to \lim\limits_{t\to\infty}\lfloor\frac{1}{4\pi}\psi_{a,\lambda}(t)\rfloor$, where $\psi_{a,\lambda}$ solves the SDE
\begin{align}\label{eq:hardedge_sde}
d \psi_{a,\lambda} = \tfrac{\beta}{2}(a+\tfrac{1}{2})\sin(\tfrac{\psi_{a,\lambda}}{2})dt + \lambda \tfrac{\beta}{4}e^{-\beta t/8} dt+ \tfrac{\psi_{a,\lambda}}{2} dt+ 2\sin(\tfrac{\psi_{a,\lambda}}{2})dB,\quad \psi_{a,\lambda}(0)=2\pi.
\end{align}

As an application of Theorem \ref{thm:HP_2piZ}, one can study 
the asymptotics of large gap probabilities of the $\HPb$ process. For $\beta>0, \Re \delta>-1/2$  let
\[
GAP_\lambda=P(\HPb\cap [0,\lambda]=\emptyset), \qquad \lambda>0,
\]
be the probability of $\HPb$ having no points in the interval  $[0,\lambda]$. Then $GAP_\lambda$ is
the probability that $\alpha_\lambda(t)\to 0$ as $t\to \infty$. The  asymptotics of $GAP_\lambda$ as $\lambda\to \infty$ can be studied with a change of measure argument, by comparing $\alpha_\lambda$ to a similar diffusion which converges to 0 a.s. This approach was carried out  in \cite{BVBV2} for the $\Sineb$ process.  The proof in \cite{BVBV2} can be extended  to cover the $\HPb$ process with a bit of extra work, we state the result without proof.

\begin{theorem}\label{thm:largegap} Fix $\beta>0$ and $\delta\in \CC$ with $\Re \delta>-1/2$. Then as $\lambda\to\infty$ we have
	\[
	GAP_\lambda = (\kappa_{\beta,\delta}+o(1))\lambda^{\gamma_{\beta,\delta}}\exp\Big(-\tfrac{\beta}{64}\lambda^2+\big(\tfrac{\beta}{8}-\tfrac{1}{4}+\tfrac{1}{2}\Im\delta\big)\lambda\Big),
	\]
	where 
	\[
	\gamma_{\beta,\delta} = \tfrac{1}{4}\big(\tfrac{\beta}{2}-\tfrac{2}{\beta}-3\big)-\Re\delta+\tfrac{2}{\beta}
	\Re(\delta+\delta^2).
	% \frac{\beta^2-8\beta\Re\delta-6\beta+4+16[(\Re\delta)^2-(\Im\delta)^2+Re\delta]}{8\beta}
	\]
\end{theorem}
A similar type of result is proved in \cite{Ramirez2011} for the asymptotic gap probability of the hard edge process. Note also that for the square root of the hard edge process, various properties (for example a transition to $\Sineb$ process and a Central Limit Theorem) have been proved by Holcomb \cite{H2018} by analyzing the coupled system of SDEs \eqref{eq:hardedge_sde}. With small modifications of the proofs therein, we get similar results for the $\HPb$ process. Again we will only record the statements without proofs.

\begin{theorem}\label{thm:CLT_transition}
	Fix $\beta>0$ and $\delta\in \CC$ with $\Re \delta>-1/2$. Then as $\lambda\to \infty$, we have
	\begin{align*}
	(\HPb-\lambda)\Rightarrow \Sineb.
	\end{align*}
	Let $N(\cdot)$ be the counting function of the $\HPb$ process, as $\lambda\to\infty$ we have
	\begin{align*}
	\tfrac{1}{\sqrt{\log \lambda}}(N(\lambda) -\tfrac{\lambda}{2\pi})\Rightarrow \mathcal{N}(0,\tfrac{2}{\beta\pi^2}),
	\end{align*}
where $\mathcal{N}(\mu,\sigma^2)$ is the mean $\mu$, variance $\sigma^2$ normal distribution. 
\end{theorem}

\subsection{Characterization of the limiting random analytic functions}

\begin{theorem}[Characterization of $\zeta^{\ttHP}_{\beta, \delta}$]\label{thm:zetaHP}
Fix $\beta>0$ and $\delta\in \CC$ with $\Re \delta>-1/2$.
Let $B_1, B_2$ independent copies of two-sided Brownian motion, and let $q$ be an independent random variable with distribution $P_{IV}(\Re \delta+1,-2 \Im \delta)$. Denote by $\HPb$ the spectrum of the operator $\Huaop$, and by $\zeta^{\ttHP}_{\beta, \delta}$ its secular function. Then $\zeta^{\ttHP}_{\beta, \delta}$ has the same distribution as the random analytic function $[1,-q]\cH_0$ where $\cH_u(z)$ is the unique analytic solution of the system of stochastic differential equations 
\begin{equation}\label{eq:HP_cH}
	d\mathcal{H} = \begin{pmatrix}
	0 & -dB_1\\0& dB_2
	\end{pmatrix}\mathcal{H}+ \begin{pmatrix}
	0 & -\Im\delta du\\0& -\Re\delta du
	\end{pmatrix}\mathcal{H} -z\frac{\beta}{8}e^{\beta u/4}J\mathcal{H} du, \qquad u\in \RR
\end{equation}
with the boundary condition $\lim\limits_{u\to -\infty} \sup_{|z|<1} \left|\cH_u(z)-\binom{1}{0}\right|=0$. Moreover, $\zeta^{\ttHP}_{\beta, \delta}(z)$ has the same distribution as the random power series $\sum_{n=0}^\infty (\cA^{(n)}_{0}-q \cB^{(n)}_{0})z^n$ where $\cA^{(n)}_{u}, \cB^{(n)}_{u}$ are processes satisfying the recursion 
\begin{align}\label{eq:HP_taylor}
    \cB^{(n)}_{u}&=-e^{B_2(u)-(\frac12+\Re\delta)u}\int_{-\infty}^u \tfrac{\beta}{8}\,e^{-B_2(s)+(\frac{\beta}{4}+\frac12+\Re\delta)s}\,\cA^{(n-1)}_{s}\,ds,\\
\cA^{(n)}_{u}&=\int_{-\infty}^u \big(\tfrac{\beta}{8}\, e^{\beta s/4}\,\cB^{(n-1)}_{s}-\Im\delta \cB^{(n)}_{s}\big)\,ds- \int_{-\infty}^u\cB^{(n)}_{s} \, dB_1.
\label{eq:HP_taylor_2}
\end{align}
with $\cA^{(0)}\equiv 1, \cB^{(0)} \equiv 0$. 
\end{theorem}

\begin{theorem}[Characterization of $\zeta^{\ttB}_{\beta,a}$]\label{thm:zetaB}
Fix $\beta>0$, $a>-1$. Let $B$ be a two-sided  Brownian motion on $\R$, $ y(t) = \exp(-\frac{\beta}{4}(2a+1) t+B(2t))$ and $\hat y(t)=y(\tfrac{4}{\beta} \log t)$. Denote by $\BB_{\beta,a}$ the spectrum of the operator $\Bessop$, and by $\zeta^{\ttB}=\zeta^{\ttB}_{\beta,a}$ its secular function. 
Then  $\zeta^{\ttB}$ has the same distribution as $
 		1+\sum_{k=1}^{\infty }r_kz^{2k},
$
 	where
 	\begin{align}\label{eq:B_Taylor}
 	r_{k}=(-1)^k2^{-2k}\iiint\limits_{0<s_1<s_2<\cdots<s_{2k}\leq 1} \frac{\hat y(s_2)\hat y(s_4)\cdots \hat y(s_{2k})}{\hat y(s_1)\hat y(s_3)\cdots \hat y(s_{2k-1})} ds_1\cdots ds_{2k}.
 	\end{align}
 Moreover, $\zeta^{\ttB}(z)$ has the same distribution as 	
 $[1,0]\mathcal{H}_0(z)$, where  $\mathcal{H}_u(z)$ is the unique strong solution of the SDE
 	\begin{equation}\label{eq:B_cH}
 	d\mathcal{H} = \begin{pmatrix}
 	0 & 0\\0& \sqrt{2}dB+(1-\frac{\beta}{4}(2a+1))du
 	\end{pmatrix}\mathcal{H}-z\frac{\beta}{8}e^{\beta u/4}J\mathcal{H} du
 	\end{equation}
 with boundary conditions $\lim\limits_{u\to -\infty} \sup_{|z|<1} \left|\cH_u(z)-\binom{1}{0}\right|=0$. 
%  The Taylor coefficients $r_k, k\ge 1$ can also be represented using can also be
%  Consequently, $\zeta_{B}(z)$ has the same distribution as the random power series $\sum_{k=0}^\infty \cA^{(2k)}_{0}z^{2k}$ where $\cA^{(2k)}_{u}$ are processes satisfying the recursion 
%  \begin{align}\label{eq:B_taylor_1}
%  \cB^{(2k+1)}_{u}&=-e^{B(2u)-\frac{\beta}{4}(2a+1)u}\int_{-\infty}^u \tfrac{\beta}{8}\,e^{-B(2s)+\frac{\beta}{2}(a+1)s}\,\cA^{(2k)}_{s}\,ds,\\
% \cA^{(2k)}_{u}&=\int_{-\infty}^u \tfrac{\beta}{8}\, e^{\beta s/4}\,\cB^{(2k-1)}_{s}\,ds,\label{eq:B_taylor_2}
%  \end{align}
%  with $\cA^{(0)}\equiv 1, \cB^{(0)} \equiv 0$. 
 \end{theorem}

\begin{remark}
The random analytic function $\zeta^{\ttB}_{\beta,a}$ can also be represented in a product form as follows:
\begin{align}\label{eq:Bess_prod}
    \zeta^{\ttB}_{\beta,a}(z)=\lim_{r\to \infty} \prod_{\substack{\lambda\in \BB_{\beta,a},\\ 0<\lambda<r}} \left(1-\frac{z^2}{\lambda^2}\right).
\end{align}
This follows from definition \eqref{eq:zeta}, the fact that the integral trace of the operator $\Bessop$ is zero, 
and the fact that point process $\BB_{\beta,a}$ is symmetric about 0.  

The random analytic function $\zeta^{\ttHP}_{\beta, \delta}(z)$ should also have a similar representation in terms of its zeros, it should be equal to the principal value product
\begin{align}\label{eq:HP_prod}
    \lim_{r\to \infty} \prod_{\substack{\lambda\in \HPb,\\ |\lambda|<r}}\left(1-\frac{z}{\lambda}\right).
\end{align}
For $\delta=0$ this statement was proved in \cite{BVBV_szeta}. Using the results of the current paper one should be able to extend the proof in \cite{BVBV_szeta} for the general $\delta$ case.
\end{remark}

\section{Convergence of discrete Dirac operators} \label{s:tools}

This section collects some of the tools that will be used to prove Theorems \ref{thm:CJ_op_limit} and \ref{thm:RO_op_limit}.
We first prove a general convergence result for the resolvents and integral traces  of  Dirac operators where the driving paths converge pointwise and are also `regular' in a certain sense.  Then we review some probabilistic tools: a standard result on the convergence of Markov chains to diffusions, and an iterated logarithm type result for products of independent random variables. 

\subsection{Convergence of resolvents and secular functions of Dirac operators}
The following proposition gives a sufficient condition for the convergence of the resolvents and integral traces of deterministic Dirac operators.

\begin{proposition}\label{prop:detlim}
	Suppose that the Dirac operators $\tau^{(n)}, n\in \ZZ_{+}\cup \{\infty\}$ are parametrized by paths $x^{(n)}+i y^{(n)}$  and boundary conditions $\uu_0=[1,0]^t$, $\uu_{1}^{(n)}=[-q^{(n)},-1]^t$. Introduce the notation $\lfloor t\rfloor_n = \lfloor nt\rfloor/n$ with the understanding that $\lfloor t\rfloor_\infty=t$. Assume that 
	there are constants $c_1, c_2>-1$, $c_3>0$, and $\kappa>0$ so that the following bounds hold for all $0\le t<1$,
	    \begin{align}\label{ineq:HS}
	    \kappa^{-1} (1-\lfloor t\rfloor_n)^{c_2}\le y^{(n)}(t)\le \kappa(1-\lfloor t\rfloor_n)^{c_1}, \qquad 
	 |q^{(n)}-x^{(n)}(t)|\le \kappa (1-\lfloor t\rfloor_n)^{c_3}
	    \end{align}
	   uniformly in $n\in \ZZ_{+}\cup\{\infty\}$ with
	    \begin{align}\label{eq:c_bounds}
	    c_3>c_2-1, \qquad c_1>c_2-2.
	    \end{align}
	 Assume that $x^{(n)}+i y^{(n)}\to x^{(\infty)}+i y^{(\infty)}$ point-wise on $[0,1)$. 
	  
	 Then
	 \begin{align}\label{DirConv}
	 \lim_{n\to \infty}    \| \ttr \tau^{(n)}-\ttr \tau^{(\infty)}\|_{HS}=0,\quad\text{and} \quad
	 \lim_{n\to \infty}|\mathfrak{t}_{\tau^{(n)}}- \mathfrak{t}_{\tau^{(\infty)}}|= 0.
	 \end{align}

\end{proposition}

\begin{proof}
From the second inequality of \eqref{ineq:HS} and the triangle inequality we have $q^{(n)}\to q^{(\infty)}$.

Denote by $R^{(n)}$ the weight function of $\tau^{(n)}$, and by $X^{(n)}$ the $2\times 2$ matrix defined in \eqref{eq:R}.	Recall that $\ttr \tau^{(n)},n\in\ZZ_{+}\cup\{\infty\}$ is an integral operator with kernel	given by \eqref{Dir:inverse}. From $q^{(n)}\to q^{(\infty)}$ and the pointwise convergence of $x^{(n)}+i y^{(n)}$ we get the pointwise convergence of the integral kernels of $\ttr \tau^{(n)}$ on $[0,1)^2$.

The bounds \eqref{ineq:HS} and the conditions on the constants $c_1, c_2, c_3$ provide  integrable upper bounds for  the functions \begin{align*}
\tr K_{\ttr \tau^{(n)}}(s,s)&= \uu_0^t R^{(n)}(s) \uu_{1}^{(n)}=\frac{x^{(n)}(s)-q^{(n)}}{2 y^{(n)}(s)}, 
\\
\tr K_{\ttr \tau^{(n)}}(s,t)^t K_{\ttr \tau^{(n)}}(s,t)&=\frac{1}{4} \|X^{(n)}(s\vee t) \uu_{1}^{(n)}\|^2\|X^{(n)}(s\wedge t) \uu_0\|^2\\
&=\frac14 \left(\frac{|q^{(n)}-x^{(n)}(s\vee t)|^2}{ y^{(n)}(s\vee t)^2}+1\right) \frac{y^{(n)}(s\vee t)}{y^{(n)}(s\wedge t)},
\end{align*}
on $[0,1)$ and $[0,1)^2$, respectively. This shows that condition \eqref{eq:int_assumption} is satisfied for  $\tau^{(n)}$ for each $n\in \ZZ_{+}\cup\{\infty\}$. Moreover, the General Dominated Convergence Theorem (see e.g. Theorem 1.4.19 in \cite{Royden}) and the point-wise convergence of the kernels lead to  \eqref{DirConv}.
\end{proof}

As an immediate consequence we have the following corollary for random Dirac operators.

\begin{corollary}\label{cor:randomDirac}
     	Suppose that $\tau^{(n)}, n\in \ZZ_{+}\cup \{\infty\}$ are random Dirac operators built from the processes $x^{(n)}+i y^{(n)}$, and boundary conditions $\uu_0=[1,0]^t$ and  
     	$\uu_{1}^{(n)}=[-q^{(n)},-1]$, with random variables $q^{(n)}$. 
     	Assume that the following  conditions are satisfied:
     	
     	\begin{enumerate}
     	    \item  $x^{(n)}+i y^{(n)}\to x^{(\infty)}+i y^{(\infty)}$ 
        in distribution  on $[0,1)$ with respect to the Skorohod topology.
     	\item There exists constants 
      	$c_1, c_2>-1$, $c_3>0$ satisfying \eqref{eq:c_bounds},
		and a sequence of tight positive random variables $\kappa^{(n)}, n\in \ZZ_{+}\cup \{\infty\}$
		so that for $0\leq t<1$
		\begin{align}\label{eq:xyfluc_bound}
		 (\kappa^{(n)})^{-1} (1-\lfloor t\rfloor_n)^{c_2}\le y^{(n)}(t) &\leq \kappa^{(n)}(1-\lfloor t\rfloor_n)^{c_1},\\ 
		|q^{(n)}-x^{(n)}(t)|&\le \kappa^{(n)} (1-\lfloor t\rfloor_n)^{c_3}\label{eq:xyfluc_bound_x}.
		\end{align}
	
     	\end{enumerate}
     	
		Then there is a coupling of $\tau^{(n)}, n\in \ZZ_+\{\infty\}$ so that 
		almost surely both $\| \ttr \tau^{(n)}-\ttr \tau^{(\infty)}\|_{HS}$ and $|\mathfrak{t}_{\tau^{(n)}}- \mathfrak{t}_{\tau^{(\infty)}}|$ converge to $0$ as $n\to \infty$.

\end{corollary}

\begin{proof}
We will show that the quadruple $(x^{(n)}+i y^{(n)}, q^{(n)},  \ttr \tau^{(n)}, \mathfrak{t}_{\tau^{(n)}})$  converges jointly in distribution to $(x^{(\infty)}+i y^{(\infty)}, q^{(\infty)},  \ttr {\tau^{(\infty)}}, \mathfrak{t}_{\tau^{(\infty)}})$ in the appropriate product space. Since both the space of cadlag functions on $[0,1)$ under the Skorohod topology and the space of $L^2$ bounded integral operators on $\R^2$ 
%with respect to the Hilbert-Schmidt norm 
are separable, the statement follows by  Skorohod's representation theorem (see e.g. Theorem 1.6.7 in \cite{Billingsley}). 

 We have to show that for any subsequence $n_j, j\in \ZZ_+$ we can choose a further subsequence $n_{j(m)}$ along which the appropriate convergence in distribution holds. 
  By the tightness of $\kappa^{(n)}, n\in \ZZ_+$ we may choose $n_{j(m)}$ so that $(x^{(n_{j(m)})}+i y^{(n_{j(m)})}, \kappa^{(n_{j(m)})})\Rightarrow (x^{(\infty)}+i y^{(\infty)}, \kappa^{(\infty)})$ with an a.s.~finite $\kappa^{(\infty)}$. 
Using Skorohod's representation theorem there is a coupling where  this convergence in distribution holds in the a.s.~sense with $x+i y$ converging pointwise on $[0,1)$. We can now use Proposition \ref{prop:detlim} to conclude that in this coupling the quadruple  $(x^{(n_{j(m)})}+i y^{(n_{j(m)})}, q^{(n_{j(m)})},  \ttr \tau^{(n_{j(m)})}, \mathfrak{t}_{\tau^{(n_{j(m)})}})$ converges a.s.~to $(x^{(\infty)}+i y^{(\infty)}, q^{(\infty)},  \ttr \tau^{(\infty)}, \mathfrak{t}_{\tau^{(\infty)}})$ in the appropriate product metric. This also implies convergence in distribution along the subsubsequence $n_{j(m)}$, finishing the proof.
\end{proof}

\subsection{Probabilistic tools}\label{subs:diff_limit}

The following two results will allow us to check the conditions in Corollary \ref{cor:randomDirac}. The first is a special case of a classical result about the diffusion limit of  discrete time Markov chains due to Ethier and Kurtz.

\begin{proposition}\label{prop:as1check}
	Suppose that for each $n\in \ZZ_+$ the the sequence of pairs of random variables  $Z_k^{(n)}=(v_k^{(n)}, w_k^{(n)})$, $0\leq k\leq n-1$ are independent. For a given $n$ let $(x_{k}^{(n)},y_{k}^{(n)}), 0\le k\le n$ be the solution of the recursion \eqref{xyrec} built from $(v_k^{(n)}, w_k^{(n)})$, 
	and introduce the notation  $(x^{(n)}(t), y^{(n)}(t)):=(x^{(n)}_{\lfloor nt\rfloor},y^{(n)}_{\lfloor nt\rfloor})$.
	
	Assume that there exist continuous  functions $a_1,a_2, \sigma_1^2, \sigma_2^2$  on $[0,1)$ such that
	\begin{align}\label{eq:as1check1}
	n\mathbb{E}(Z_{k}^{(n)})&=
	\begin{pmatrix}
	a_1(\tfrac{k}{n}) & a_2(\tfrac{k}{n})
	\end{pmatrix}+\err_1(k,n),\\
	n\mathrm{Cov}(Z_{k}^{(n)},Z_{k}^{(n)})&=
	\begin{pmatrix}
	\sigma_{1}^2(\tfrac{k}{n}) & 0\\ 0& \sigma_{2}^2(\tfrac{k}{n})
	\end{pmatrix}+\err_2(k,n), 
	\end{align}
	and
	\begin{equation}\label{eq:as1check2}
	n\mathbb{E}( |v_{k}^{(n)}|^4+|w_{k}^{(n)}|^4)=\err_3(k,n),
	\end{equation}
	where the error terms satisfy
	\[
	\limsup_{n\to \infty} \max_{k/n\le 1-\delta}|\err_j(k,n)|=0
	\]
	for any $\delta\in (0,1)$, $1\le j\le 3$.

	Then $x^{(n)}+iy^{(n)}$ converges in distribution to $x+iy$, the solution of the stochastic differential equation
	\begin{equation}\label{eq:as1sde}
	dx = (a_1(t)dt+\sigma_1(t)dB_1)y,\quad  dy=(a_2(t)dt+\sigma_2(t)dB_2)y,\quad x(0)=0,y(0)=1,
	\end{equation} 
	on $[0,1)$ with respect to the Skorohod topology. Here $B_1$ and $B_2$ are independent standard Brownian motion. 
\end{proposition}
\begin{proof}
	The  proposition follows from Theorem 7.4.1 and Corollary 7.4.2 of \cite{EthierKurtz} (see Section 11.2 in \cite{SV} as well).
	% The time homogeneity required in the Theorem 7.4.1 and Corollary 7.4.2 of \cite{EthierKurtz} can be generalized by introducing a time variable. And by Theorem 5.3.7 of \cite{EthierKurtz}, the solution to \eqref{eq:as1sde} is unique. Thus, we only need to check the assumptions of Corollary 7.4.2. Notice that the step increment $X_{n,k+1}-(x,y)^t$ given $X_{n,k}=(x,y)^t$ is equal to $yY_{n,k}$. The assumptions (4.14) and (4.15) of Corollary 7.4.2 are satisfied by our condition \eqref{eq:as1check1}. The supremum bound (4.16) follows from \eqref{eq:as1check2} and Markov's inequality.  Thus we conclude the convergence by the Corollary 7.4.2.
\end{proof}
% \begin{remark}
% 	It suffices to study the random walk $ \sum_{l=0}^{k-1}\log(1+w_{n,l})$ in the case $x_{n,k}\equiv 0$. Let $\xi_{n,k}=\mathbb{E}[\log(1+w_{n,k})]$. Assuming that
% 	\[
% 	\sum_{l=0}^{k-1} \xi_{n,k} = f(\frac{k}{n})+o(1),\quad n\mathrm{Var}[\log(1+w_{n,k})] = g(\frac{k}{n})+o(1),\quad n\mathbb{E}[(\log 1+w_{n,k}-\xi_{n,k})^4] = o(1)
% 	\]
% 	uniformly for $k/n$ in compact sets of $[0,1)$. Again by Theorem 7.4.1 and Corollary 7.4.2 of \cite{EthierKurtz}, there exists a standard Brownian motion $B$ such that 
% 	\[
% 	y_{n}(t):=y_{n,\lfloor nt\rfloor} \Rightarrow \exp(f(t)+\int_0^t \sqrt{g(s)}dB_s).
% 	\]

% \end{remark}

Our next statement  provides a sufficient condition to check the inequality  \eqref{eq:xyfluc_bound} for our models. 
% will discuss a general tightness result on the $y$ coordinate. Notice that $y_{n,k}=\prod_{j=0}^k (1+w_{n,k})$, it is enough to bound the sum $\sum_{j=0}^{k}\log(1+w_{n,k})$. 
The  proposition is a straightforward extension of  Lemma 5 of \cite{RR}, we do not present the proof here. (See (2.4)-(2.5) of Lemma 5 and also Claim 10 in \cite{RR}.)
\begin{proposition}\label{prop:ygeneral}
	Let $\xi^{(n)}_{k},0\leq k\leq n-1, 1\le n$ be a positive triangular array with independent entries for any given $n$. Define $y_{j}^{(n)}=\prod_{k=0}^{j-1}\xi^{(n)}_{k}$. Assume that there are constants  $\lambda_0>0$, $c_1\in\R$ and $c_2,c_3>0$, 	so that for $|\lambda|<\lambda_0$ and  $0\leq j\leq n-1$ we have
	\begin{align}\label{eq:mgfA_n}
	\log \ev [\exp( \lambda \log y^{(n)}_{j})] = c_1\lambda\log(1-\tfrac{j}{n})-c_2\lambda^2\log(1-\tfrac{j}{n}) + \err_n(j),
	\end{align} 
	where $|\err_n(j)|\leq c_3$  for all $j, n$. Then for any $\eps>0$ small, there exists a sequence of tight positive random variables $\kappa^{(n)}=\kappa^{(n)}(\eps)$ such that for all $0\leq k \le n-1$ we have
	\begin{align*}
	(\kappa^{(n)})^{-1}(1-\tfrac{k}{n})^{c_1+\eps} \le y^{(n)}_{k}\le \kappa^{(n)}(1-\tfrac{k}{n})^{c_1-\eps}.
	\end{align*}
\end{proposition}

% \begin{proof}
% 	Consider the recentered exponent $A_{n,j}=\log y_{j}^{(n)}-c_1\log(1-\frac{j}{n})$. For $\lambda\in (-\lambda_0,\lambda_0)$ define  $Z_{n,j,\lambda}=\exp(\lambda A_{n,j}) \mathbb{E}[\exp(\lambda A_{n,j})]^{-1}$. Then $\{Z_{n,j, \lambda},1\leq j\leq n-1\}$ is a martingale adapted to the filtration generated by $\{\xi_{n,k},0\leq k\leq j-1\}$. By Doob's inequality and \eqref{eq:mgfA_n}, for $b>0$ we have
% 	\[
% 	\mathbb{P}\big(\sup_{0\leq t<j/n} ( A_{n}(t) -c_2\lambda v(t)+\lambda^{-1}c_3) \geq b \big)\leq e^{-\lambda b} .
% 	\]
% 	Set $v(t)=-\log(1-t)$ on $[0,1)$, by choosing $\lambda$ and $b$ appropriately, we can conclude that the random variables 
% 	\[
% 	A^*_n:=\sup_{0\leq j\leq n-1} |A_{n,j}|/v^{3/4}(\frac{j}{n})  \quad \text{ are tight in distribution.}
% 	\]
% 	(Any constant in $(1/2,1)$ would work in place of $3/4$ here.) From this it follows that for any $\eps >0$, we can find a sequence of tight random variables $C_n(\eps)$ such that 
% 	\begin{align}\label{ineq:A_diff_abs}
% 	|A_{n,j} - A_{n,i}|\leq C_n + \eps v(\frac{j}{n}) + \eps v(\frac{i}{n})
% 	\end{align}
% 	for all $0\leq i< j\leq n-1$. Consider the case when $i=0$, the two-sided bounds on $y_{n,j}$ follows by choosing $\kappa_n=\exp(C_n)$.
% \end{proof}

\section{Path convergence for the discrete models} \label{s:pathconv}
In this section, we prove that the driving paths of the operators $\CJop$ and $\ROop$ converge in distribution to the driving paths of the operators $\Huaop$ and $\Bessop$, respectively. For this we will check that the discrete models satisfy the conditions in Proposition \ref{prop:as1check}. 
\subsection{Circular Jacobi ensemble}

Recall the definition of the distributions $\Theta(a+1,\delta)$ and $P_{IV}(m,\mu)$ from Definitions \ref{def:Theta}  and \ref{def:Pearson}. We also introduce an additional distribution. 

\begin{definition}
For $s,t>0$ let $\mathrm{B}'(s,t)$ denote the `beta prime' distribution on $(0,\infty)$  that has the probability density function 
	\begin{align*}
\tfrac{\Gamma(s+t)}{\Gamma(s)\Gamma(t)}y^{s-1}(1+y)^{-s-t}.
	\end{align*} 

\end{definition}
Note that if $X_i, i=1,2$ are independent Gamma distributed random variables with density $\Gamma(\alpha_i)^{-1} x^{\alpha_i-1} e^{-x}$ on $(0,\infty)$ then $\frac{X_1}{X_2}$ has $\mathrm{B}'(\alpha_1,\alpha_2)$ distribution, and $\frac{X_2-X_1}{X_1+X_2}$ has $\mathrm{\tl B}(\alpha_1,\alpha_2)$ distribution.

The following statement follows by a simple change of variables. 

\begin{fact}\label{fact:factor}
Suppose that $\gamma\in \CC$ is distributed as $\Theta(a+1,\delta)$ with $a\ge 0$ and $\Re \delta>-1/2$. Define $w, v\in \R$ with $\tfrac{2\gamma}{1-\gamma}= w-iv$. 
Then the random variables $w$ and $\frac{v}{2+w}$ are \underline{independent}, and 
\begin{align*}
    1+w \sim \mathrm{B}'(\tfrac{a}{2},\tfrac{a}{2}+2\Re \delta+1), \qquad 
    \tfrac{v}{2+w} \sim P_{IV}(\tfrac{a}{2}+\Re\delta+1,-2 \Im \delta). 
\end{align*}
In the $a=0$ case $w$ degenerates to $-1$, and hence $\frac{v}{2+w}=v$.
\end{fact}

We record here the following facts of the beta prime and Pearson type $\textrm{IV}$ distributions.
\begin{fact}\label{fact:moments}
	Let $s, t>0$, and $Y\sim \mathrm{B}'(s,t)$. Then for any $-s<k<t$,
		\begin{align*}
	\mathbb{E}[Y^k] = \tfrac{\Gamma(s+k)\Gamma(t-k)}{\Gamma(s)\Gamma(t)}.
	\end{align*}
	Let $m>5/2$, $\mu\in \R$, and $Z\sim P_{IV}(m,\mu)$. Then we have
	\begin{align*}
	    \ev[Z]=-\tfrac{\mu}{2m-2}, \qquad \ev[Z^2]=\tfrac{2m-2+\mu^2}{(2m-2)(2m-3)},\qquad 	\ev[Z^4]=\tfrac{12 (m+(\mu^2-3)/2)^2-2 \mu ^4 -2\mu^2-3}{(2m-5)(2m-4)(2m-3)(2m-2)}. 
	\end{align*}
\end{fact}

We are now ready to prove that the driving paths of the operators $\CJop$  converge to the driving path of the operator $\Huaop$.

\begin{proposition}\label{prop:cjpath} Fix $\beta>0$ and $\delta\in \CC$ with $\Re\delta>-1/2$.
	Let $\{\gamma_{k}^{(n)},0\leq k\leq n-1\}$ be random variables 
	that are independent for a fixed $n$, and have distributions $\gamma^{(n)}_{k}\sim \Theta(\beta(n-k-1)+1,\delta)$. Define $v^{(n)}_{k}, w^{(n)}_{k}\in \R$ via \eqref{def:wv} using $\gamma_k=\gamma^{(n)}_{k}$, and let $x^{(n)}_k, y^{(n)}_k, 0\le k\le n$ be the solution of the recursion \eqref{xyrec} using $v_k=v^{(n)}_{k}$, $w_k=w^{(n)}_{k}$.  Set $(x^{(n)}(t), y^{(n)}(t)):=(x^{(n)}_{\lfloor nt\rfloor},y^{(n)}_{\lfloor nt\rfloor})$.
	 Let $\tl x+\tl y$ be the  process defined in Proposition \ref{prop:hpoperator}. Then  $x^{(n)}+iy^{(n)}$ converges in distribution to $\tl x+ i \tl y$  on $[0,1)$ with respect to the Skorohod topology. 
\end{proposition}
\begin{proof}
Let $N_{\delta}=\lceil\frac2\beta(2-\Re\delta)\rceil\vee 0$.
% and $c_{\delta}=\frac{2}{\beta}(2\Re\delta+1)>0$.

Set $z^{(n)}_{k}=v^{(n)}_{k}/(2+ w^{(n)}_{k})$. By Fact \ref{fact:factor} we have that $1+w^{(n)}_{k}$ and $z^{(n)}_{k}$ are independent with distributions 
\begin{align}\label{eq:w_dist}
1+w^{(n)}_{k}&\sim \mathrm{B}'(\tfrac{\beta}{2}(n-k-1),\tfrac{\beta}{2}(n-k-1)+2\Re \delta+1), \\ z^{(n)}_{k}&\sim P_{IV}(\tfrac{\beta}{2}(n-k-1)+\Re\delta+1,-2 \Im \delta).\label{eq:z_dist}
\end{align}
From  Fact \ref{fact:moments}, we get that for $0\leq k\leq n-N_\delta-1$ 
	\begin{align}
	\mathbb{E}[w^{(n)}_{k}] = \tfrac{-4\Re\delta}{\beta(n-k-1)+4\Re\delta},  \qquad \mathbb{E}[(w^{(n)}_{k})^2]=\tfrac{4\beta(n-k-1)-8\Re\delta+16(\Re\delta)^2}{(\beta(n-k-1)+4\Re\delta-2)(\beta(n-k-1)+4\Re\delta)}, \label{eq:wmoments}\\
	\mathbb{E}[v^{(n)}_{k}]= \tfrac{4\Im\delta}{\beta(n-k-1)+4\Re\delta},\qquad\mathbb{E}[(v^{(n)}_{k})^2]= \tfrac{4\beta(n-k-1)+8\Re\delta+16(\Im\delta)^2}{(\beta(n-k-1)+4\Re\delta-2)(\beta(n-k-1)+4\Re\delta)}\label{eq:vmoments}.
	\end{align}
	Moreover, there exists a constant $c>0$ so that for $0\le k \le n-N_\delta-1$ we have
	\[
	\left|\ev[v^{(n)}_{k}w^{(n)}_{k}]\right|+\ev[(v^{(n)}_{k})^4]+\ev[(w^{(n)}_{k})^4]\le c (n-k)^{-2}.
	\]
% 	For fixed $k$, the random variables $w_{n,k}$ and $v_{n,k}$ are asymptotically independent, i.e. as $n\to\infty$,
% 	\begin{align}\label{eq:asympindep}
% 	\mathbb{E}[v_{n,k}w_{n,k}] = \frac{2\Im\delta(4-8\Re\delta)}{\beta^2(n-k-1)^2} +O((n-k)^{-3}).
% 	\end{align}
% 	Moreover, we have $\mathbb{E}[v_{n,k}^4]=\frac{48}{\beta^2}(n-k)^{-2}+O((n-k)^{-3})$, and $\mathbb{E}[w_{n,k}^4]=\frac{48}{\beta^2}(n-k)^{-2}+O((n-k)^{-3})$ as $n\to\infty$.
	This means that 
	the conditions \eqref{eq:as1check1} and \eqref{eq:as1check2} of Proposition \ref{prop:as1check} are satisfied with the functions $a_1(t)=\Im\delta \upsilon_\beta'(t)$, $a_2(t)=-\Re\delta \upsilon_\beta'(t)$, $\sigma_{1}^2(t)=\sigma_{2}^2(t)=\upsilon_\beta'(t)$, with $\upsilon_\beta(t)=-\tfrac{4}{\beta}\log(1-t)$.
	Hence the processes $x^{(n)}(t)+i y^{(n)}(t)$ converge in distribution to the solution of the sde
	\begin{equation}\label{eq:sde_99}
	dx=\left(\Im\delta \upsilon_\beta'(t) dt +\sqrt{\upsilon_\beta'(t)}dB_1\right) y, \quad 		dy=\left(-\Re\delta \upsilon_\beta'(t) dt +\sqrt{\upsilon_\beta'(t)}dB_2\right) y 
	\end{equation}
	with independent Brownian motions $B_1,B_2$ and initial values $x(0)=0,y(0)=1$.  The distribution of the process in \eqref{eq:sde_99}  is the same as that of the SDE \eqref{eq:HPsde} with the time change $t\to \upsilon_\beta(t)$, which is completes the proof of the proposition.
\end{proof}

\subsection{Real orthogonal ensemble}
Now we turn to the path convergence of the real orthogonal ensemble. By Theorem \ref{thm:OrthVer}, the modified Verblunsky coefficients of the real orthogonal ensemble are all real. Hence \eqref{def:wv} and \eqref{xyrec} imply that $v_k=x_k=0$, $1+w_k=\tfrac{1+\gamma_k}{1-\gamma_k}$, and $y_k=\prod_{j=0}^{k-1} \tfrac{1+\gamma_k}{1-\gamma_k}$.

% By Theorem \ref{thm:OrthVer}, the Verblunsky coefficients corresponding to the real orthogonal ensemble are real, hence
% \[
% \gamma_{2n,k}=\alpha_{2n,k},\quad v_{2n,k}=0,\quad w_{2n,k}=\frac{2\alpha_{2n,k}}{1-\alpha_{2n,k}} \text{\quad for  $0\le k\le 2n-2$},
% \]
% and $\alpha_{2n,2n-1}=-1$. Thus, $x^{(n)}(t)= 0$ and $y^{(n)}(t)=y_{2n,\lfloor 2nt\rfloor} = \prod_{k=0}^{\lfloor 2nt\rfloor-1} (1+w_{2n,k})$ for all $0\le t<1$. Note that $(1+w_{n,k})$ is distributed according to beta prime distribution. By Fact \ref{fact:moments}, and Proposition \ref{prop:as1check}, we have the following path convergence. 

\begin{proposition}\label{prop:orthpath}
	 Fix $a,b>-1,\beta>0$. Let $\{\gamma^{(2n)}_{k},0\leq k\leq 2n-1\}$ be  random variables that are independent for a fixed $n$ with the following distributions: $\gamma^{(2n)}_{2n-1}=-1$, and for $0\le k\le 2n-2$
	 \begin{align}
	  \gamma^{(2n)}_{k}\sim \begin{cases}
	    \mathrm{\tl B}\big(
	\tfrac{\beta}{4}(2n-k+2a)
, \tfrac{\beta}{4}(2n-k+2b)
\big),\quad &\text{if $k$ is even,}\\
\mathrm{\tl B}\big(\tfrac{\beta}{4}(2n-k+2a+2b+1),
	\frac{\beta}{4}(2n-k-1)
	\big),\quad &\text{if $k$ is odd.}
	    \end{cases}
	 \end{align}	
	 Define $y^{(2n)}(t)= \prod_{k=0}^{\lfloor 2nt\rfloor-1}\frac{1+\gamma^{(2n)}_{k}}{1-\gamma^{(2n)}_{k}}$ for all $0\le t<1$. Let $\tl y$ be the process  defined in Proposition \ref{prop:hardedge}. Then  $y^{(2n)}$ converges in distribution to  $\tl y$ on  $[0,1)$ with respect to the Skorohod topology.
\end{proposition}
\begin{proof}
	We first consider the multiplicative random walk with step size $2$ and define $y_1^{(2n)}(t):=\prod_{k=0}^{2\lfloor nt\rfloor-1}\frac{1+\gamma^{(2n)}_{k}}{1-\gamma^{(2n)}_{k}}$. 
	We will check  the conditions in Proposition \ref{prop:as1check} for $y_1^{(2n)}(t)$ (with $x_1^{(2n)}=0$).

If $\gamma\sim \mathrm{\tl B}(s_1,s_2)$ then $\frac{1+\gamma}{1-\gamma} \sim \mathrm{B}'(s_2,s_1)$. Using the moment formulas of Fact \ref{fact:moments} one readily checks that with 
\[
v_k^{(2n)}=0, \qquad w_k^{(2n)}=\tfrac{1+\gamma^{(2n)}_{2k}}{1-\gamma^{(2n)}_{2k}}\cdot \tfrac{1+\gamma^{(2n)}_{2k+1}}{1-\gamma^{(2n)}_{2k+1}}-1
\]
the conditions  \eqref{eq:as1check1} and \eqref{eq:as1check2} of Proposition \ref{prop:as1check} are satisfied with $a_1=\sigma_{1}^2=0$, $a_2(t)=\frac{4/\beta-(2a+1)}{(1-t)}$ and $\sigma_{2}^2(t)=\frac{8}{\beta(1-t)}$. Hence the limit in distribution of $y_1^{(2n)}(\cdot)$ exist and it has the distribution of the strong solution of the diffusion
\[
d\tl y=\tfrac{4/\beta-(2a+1)}{(1-t)} \tl ydt+\sqrt{\tfrac{8}{\beta(1-t)}}\tl y dB, \qquad \tl y(0)=1,
\]
where $B$ is a standard Brownian motion. 

The solution of this SDE has the same distribution as  the process $\tl y$ in Proposition \ref{prop:hardedge}.
%Hence by  Proposition \ref{prop:as1check} we can couple the processes $y_1^{(n)}$ with the process $\tl y$ so that $y_1^{(n)}\to \tl y$ on $[0,1)$ a.s.
 Using the  the fourth moment bounds of Fact \ref{fact:moments} one can show that $|{y_1^{(2n)}}/{y^{(2n)}}-1|$ converges to 0 in the sup-norm in probability on any compact subset of $[0,1)$. From this it follows that that $y^{(2n)}$ converges to $\tl y$ in distribution as well, proving the proposition. 
\end{proof}

\section{Proofs of the operator limit theorems} \label{s:oplim}

We are ready to prove Theorem \ref{thm:CJ_op_limit}. We will do that by applying Corollary \ref{cor:randomDirac} to the processes described in  Propositions \ref{prop:cjpath}, for this we need 
to prove the path bounds \eqref{eq:xyfluc_bound} and \eqref{eq:xyfluc_bound_x}. This is the content of Propositions \ref{prop:y_cj} and \ref{prop:xtight} below.

\begin{proposition}\label{prop:y_cj}
Fix $\beta>0$, $\delta\in\CC$ with $\Re\delta>-1/2$. Let $x^{(n)}_{k}+i y^{(n)}_{k}, 0\le k\le n$ be  defined as in Proposition \ref{prop:cjpath}. Then for any $0<\eps<c_\delta=\frac{4}{\beta}(\Re\delta+\frac12)$, there exists a sequence of tight random variables $\kappa^{(n)}=\kappa^{(n)}(\eps)$ such that for all $0\le k\le n-1$,
\begin{align}\label{eq:ynk_bound}
(\kappa^{(n)})^{-1}(1-\tfrac{k}{n})^{c_\delta+\eps}\le y^{(n)}_{k}\le \kappa^{(n)}(1-\tfrac{k}{n})^{c_\delta-\eps}.
\end{align}
\end{proposition}
\begin{proof}
Using the definition of $y^{(n)}_k$ together with Fact \ref{fact:factor} we get that 
\[
y^{(n)}_k=\prod_{j=0}^{k-1} (1+w^{(n)}_{k}),
\]
where for a fixed $n$ the random variables $w^{(n)}_{k}, 0\le k\le n-1$ are independent with distribution given in \eqref{eq:w_dist}.
 By Fact \ref{fact:moments}, for $|\lambda|<\Re \delta+1/2$  and $0\le k\le n-1$ we have 
\[
	\log E[(y^{(n)}_{k})^\lambda] = \sum_{j=0}^{k-1}\log \left(\tfrac{\Gamma(s^{(n)}_{j}+\lambda)\Gamma(t^{(n)}_{j}-\lambda)}{\Gamma(s^{(n)}_{j})\Gamma(t^{(n)}_{j})}\right),
\]
where $s^{(n)}_{j}=\frac{\beta}{2}(n-j-1)$, $t^{(n)}_{j}=\frac{\beta}{2}(n-j-1)+2\Re\delta+1$. By the asymptotics of the Gamma function  for any $r>0$ there is a $c_r>0$ so that 
\[
\left|\log \Gamma(x)-\left((x-\tfrac12)\log x +x-\tfrac{\log 2\pi}{2}-\tfrac{1}{12}x^{-1}\right)\right|\le c_r x^{-2} \quad \text{for $x\ge r$.}
\]
From this (and some basic Taylor expansion estimates) it follows that $y^{(n)}_k$ satisfies condition \eqref{eq:mgfA_n} of  Proposition \ref{prop:ygeneral} with $c_1=c_\delta$ and $c_2=\frac{2}{\beta}$, and the statement follows by  Proposition \ref{prop:ygeneral}.
\end{proof}

\begin{proposition}\label{prop:xtight}
Fix $\beta>0$, $\delta\in\CC$ with $\Re\delta>-1/2$. Let $x^{(n)}_{k}+i y^{(n)}_{k}, 0\le k\le n$ be  defined as in Proposition \ref{prop:cjpath}.
	Then for any $0<c'<c_\delta=\frac{4}{\beta}(\Re\delta+\frac12)$, there exist tight random constants $\kappa^{(n)}_1>0$ such that 
	\begin{align}\label{ineq:xtight}
	|x^{(n)}_{n}-x^{(n)}_{j}|\leq \kappa^{(n)}_1(1-\frac{j}{n})^{c'}\quad \text{ for all $0\leq j\leq n-1$.}
	\end{align}
\end{proposition}

\begin{proof}
Fix $\eps>0$ so that $c'+2\eps<c_\delta$. By Proposition \ref{prop:y_cj} there is a sequence of tight random variables $\kappa^{(n)}$ so that \eqref{eq:ynk_bound} holds, and the sequence $\kappa^{(n)}$ is measurable with respect to the sigma-field generated by the random variables $y^{(n)}_k, 0\le k\le n-1$.     

Set $z^{(n)}_{k}=v^{(n)}_{k}/(2+w^{(n)}_{k})$. Then from \eqref{xyrec}  we get
\[
x^{(n)}_{k+1}=x^{(n)}_{k}+z^{(n)}_k(2+w^{(n)}_{k})y^{(n)}_{k}=x^{(n)}_{k}+z^{(n)}_{k}(y^{(n)}_{k+1}+y^{(n)}_{k}),
\]
and
\[
x^{(n)}_{n}-x^{(n)}_{j}=\sum_{k=j}^{n-1}z^{(n)}_{k}(y^{(n)}_{k}+y^{(n)}_{k+1}).
\]
Introduce 
\[
A^{(n)}:=\max_{0\leq j\leq n-1} \left|\sum_{k=j}^{n-1}z^{(n)}_{k}(y^{(n)}_{k}+y^{(n)}_{k+1})\right|(1-\frac{j}{n})^{-c'},
\]
the statement will follow once we show that the sequence $A^{(n)}, n\ge 1$ is tight. We will do that by first separating finitely many terms in the maximum, and then splitting the sum using centered versions of $z^{(n)}_{k}$.

Set $N_{\delta}=\lceil\frac2\beta(4-\Re\delta)\rceil\vee 0$ and $\tl n=n-N_\delta-1$. Note that by Fact \ref{fact:moments}, the fourth moment of $z^{(n)}_{k}$ is finite for $j\le \tl n$. 
By \eqref{eq:z_dist} the distribution of $z^{(n)}_{k}$ only depends on $n-k$, hence the path bounds \eqref{eq:ynk_bound} on $y^{(n)}_k$ (together with $c_\delta-2\eps-c'>0$) imply that the following sequence of random variables is tight:
\begin{align}\label{eq:An_0}
A^{(n)}_{0}:=\max_{\tl n+1\leq j\leq n-1} \left|\sum_{k=j}^{n-1}z^{(n)}_{k}(y^{(n)}_{k}+y^{(n)}_{k+1})\right|(1-\frac{j}{n})^{-c'}.
\end{align}
%Now we consider the maximum over $0\leq j\leq n-N_\delta-1$. 
Since the sequence $A^{(n)}_{0}, n\ge 1$ is tight, it suffices to show the tightness of the following sequence:
\begin{align}
\tl A^{(n)}:= \max_{0\leq j\leq \tl n } \left|\sum_{k=j}^{\tl n}z^{(n)}_{k}(y^{(n)}_{k}+y^{(n)}_{k+1})\right|(1-\frac{j}{n})^{-c'}.\end{align}
We introduce 
\begin{align*}
    A^{(n)}_{1}&=\max_{0\leq j\leq \tl n } \left|\sum_{k=j}^{\tl n} \ev[z^{(n)}_{k}](y^{(n)}_{k}+y^{(n)}_{k+1})\right|(1-\frac{j}{n})^{-c'},
\\
    A^{(n)}_{2}&=\max_{0\leq j\leq \tl n } \left|\sum_{k=j}^{\tl n} \bar z^{(n)}_{k}(y^{(n)}_{k}+y^{(n)}_{k+1})\right|(1-\frac{j}{n})^{-c'},
\end{align*}
where $\ol{X}=X-E[X]$. Note that $\tl A^{(n)}\le A^{(n)}_{1}+A^{(n)}_{2}$.

By \eqref{eq:z_dist} and Fact \ref{fact:moments} we have 
\[
\ev[z^{(n)}_{k}]=\frac{2\Im \delta}{\beta(n-k-1)+2\Re \delta}.
\]
Using the bounds in  \eqref{eq:ynk_bound} with $\eps<c_{\delta}-c'$ we get 
\begin{align}
    A^{(n)}_{1}&\le 
    \max_{0\leq j\leq \tl n }\left\{ (1-\frac{j}{n})^{-c'} \left(\sum_{k=j}^{\tl n} 4\kappa^{(n)} (1-\frac{k}{n})^{c_\delta-\eps} \frac{|\Im\delta|}{\beta(n-k-1)+2\Re\delta}\right)\right\}
    \le c \kappa^{(n)},
\end{align}
 with a deterministic constant $c$ that only depends on $\delta$ and $\beta$. This shows that the sequence $A^{(n)}_{1}, n\ge 1$ is tight.

Next we turn to the tightness of the sequence $A^{(n)}_{2}$.
Choose $1<\theta<(c_\delta-\tfrac32\eps)/c'$. Define \begin{align*}m&=m^{(n)}=\inf\{i\in\Z^+: \theta^i \geq \log(\tfrac{n}{N_\delta+1})\},\\
\sigma_0&=\sigma^{(n)}_0=0, \qquad \sigma_i=\sigma^{(n)}_i=\min(\lfloor n(1-e^{-\theta^i})\rfloor,\tl n) \quad \text{for $1\leq i\leq m$. }
\end{align*}
Note that $\sigma_0=0\le \sigma_1 \le \dots \le \sigma_m=\tl n$.
In order to bound the tail of $A^{(n)}_{2}$ we will  split the index set of the sums into blocks $\{\sigma_i,\sigma_i+1,\cdots,\sigma_{i+1}\}$ to control the term $(1-j/n)^{-c'}$, and then control the fluctuations within each block. Fix $K>0$, then we have
\begin{align}\label{eq:blocks}
	P(A^{(n)}_{2}\geq K)\le &
	\sum_{i=0}^{m-1} P\left(\max_{\sigma_i\le j\le \sigma_{i+1}}|\sum_{k=j}^{\tl n}\bar{z}^{(n)}_{k}(y^{(n)}_{k}+y^{(n)}_{k+1})|(1-\frac{j}{n})^{-c'}\ge K, \kappa^{(n)}\le \sqrt{K}\right)\\
	&+P(\kappa^{(n)}>\sqrt{K}).\notag
\end{align}
Since $\kappa^{(n)}$ are tight, we have 
\[
\lim_{K\to\infty}\limsup_{n\to\infty} P(\kappa^{(n)}>\sqrt{K})=0.
\]
We now estimate the terms in the sum in \eqref{eq:blocks} for each  $0\le i\le m-1$.  We have
\begin{align*}
&P\left(\max_{\sigma_i\le j\le \sigma_{i+1}}|\sum_{k=j}^{\tl n}\bar{z}^{(n)}_{k}(y^{(n)}_{k}+y^{(n)}_{k+1})|(1-\frac{j}{n})^{-c'}\ge K, \kappa^{(n)}\le \sqrt{K}\right)\\
&\qquad \qquad \leq P\big( |\sum_{k=\sigma_i}^{\tl n}\bar{z}^{(n)}_{n}(y^{(n)}_{k}+y^{(n)}_{k+1})|\ge \frac{K}{2}(1-\frac{\sigma_{i+1}}{n})^{c'}, \kappa^{(n)}\le \sqrt{K}\big)\\ &\qquad \qquad\quad +P\big(\max_{\sigma_i\leq j\leq \sigma_{i+1}}|\sum_{k=\sigma_i}^{j}\bar{z}^{(n)}_{k}(y^{(n)}_{k}+y^{(n)}_{k+1})|\geq\frac{K}{2}(1-\frac{\sigma_{i+1}}{n})^{c'},\kappa^{(n)}\le \sqrt{K}\big).
\end{align*}
Note that the sequence $\kappa^{(n)}$ is measurable with respect to $y^{(n)}_{k}, 0\le k\le n$ and $\bar z^{(n)}_{k}$ are independent of $y^{(n)}_{k}$. Hence by conditioning on $y^{(n)}_{k}, 0\le k\le n$, using Doob's maximal inequality, and the path bound \eqref{eq:ynk_bound}  we get
\begin{align*}
   & P\big(\max_{\sigma_i\leq j\leq \sigma_{i+1}}|\sum_{k=\sigma_i}^{j}\bar{z}^{(n)}_{k}(y^{(n)}_{k}+y^{(n)}_{k+1})|\geq\frac{K}{2}(1-\frac{\sigma_{i+1}}{n})^{c'},\kappa^{(n)}\le \sqrt{K}\big)\\
   &\qquad\qquad\qquad\qquad \le \ev\left[\ind(\kappa^{(n)}\le \sqrt{K}) \sum_{k=\sigma_i}^{\sigma_{i+1}}
    \frac{4 \ev[ (\bar z^{(n)}_{k})^2] (y^{(n)}_{k}+y^{(n)}_{k+1})^2}{K^2 (1-\sigma_{i+1}/n)^{2c'}}
    \right]\\
    &\qquad\qquad\qquad\qquad \le 
    \ev\left[\ind(\kappa^{(n)}\le \sqrt{K}) \sum_{k=\sigma_i}^{\sigma_{i+1}}
    \frac{16 (\kappa^{(n)})^2 \ev[ (\bar z^{(n)}_{k})^2] (1-k/n)^{2(c_\delta-\eps)}}{K^2 (1-\sigma_{i+1}/n)^{2c'}}
    \right]\\
    &\qquad\qquad\qquad\qquad \le  \sum_{k=\sigma_i}^{\sigma_{i+1}}
    \frac{16 \ev[ (\bar z^{(n)}_{k})^2] (1-k/n)^{2(c_\delta-\eps)}}{K (1-\sigma_{i+1}/n)^{2c'}}.
\end{align*}
Using \eqref{eq:z_dist} and Fact \ref{fact:moments} one can show that there exists an absolute constant $c$ such that 
\begin{align*}
    \sum_{k=\sigma_i}^{\sigma_{i+1}}
    \frac{16 \ev[(\bar z^{(n)}_{k})^2] (1-k/n)^{2(c_\delta-\eps)}}{K (1-\sigma_{i+1}/n)^{2c'}}&\le  cK^{-1}(1-\frac{\sigma_{i+1}}{n})^{-2c'}(1-\frac{\sigma_i}{n})^{2(c_\delta-\eps)}
    \le cK^{-1}e^{-2\theta^i(c_\delta-\eps-c'\theta)}\\
    &\le c K^{-1} e^{-\eps \theta^i}.
\end{align*}
Similarly,  Chebishev's inequality, conditioning, and the path bound \eqref{eq:ynk_bound}  give the upper bound
\[
P\big( |\sum_{k=\sigma_i}^{\tl n}\bar{z}^{(n)}_{k}(y^{(n)}_{k}+y^{(n)}_{k+1})|\ge \frac{K}{2}(1-\frac{\sigma_{i+1}}{n})^{c'}, \kappa^{(n)}\le \sqrt{K}\big)\le c K^{-1} e^{-\eps \theta^i}.
\]
This shows that the sum on the right of \eqref{eq:blocks} can be bounded from above by 
\[
2 \sum_{i=0}^m c K^{-1} e^{-\eps \theta^i}\le c_1 K^{-1}
\]
with an absolute constant $c_1$. This proves the tightness of the sequence $A^{(n)}_{2}, n\ge 1$, and completes the proof of the proposition. 
\end{proof}
Now we have all the pieces for the proof of Theorem \ref{thm:CJ_op_limit}.
\begin{proof}[Proof of Theorem \ref{thm:CJ_op_limit}]
Consider the random variables $x^{(n)}_{k}+i y^{(n)}_{k}, 0\le k\le n$ defined in Proposition \ref{prop:cjpath}, and define $(x^{(n)}(t), y^{(n)}(t)):=(x^{(n)}_{\lfloor nt\rfloor},y^{(n)}_{\lfloor nt\rfloor})$. Let $\tl x+i \tl y$ be the process  defined  in Proposition \ref{prop:hpoperator}. 
Set $q^{(n)}=x^{(n)}_n$ and $q=\lim\limits_{t\to 1} \tl x(t)$. Define $\tau^{(n)}, n\in \ZZ_+$ using $(x^{(n)}+i y^{(n)}, q^{(n)})$, and $\tau^{(\infty)}$ using  $(\tl x+i \tl y,q)$. Then $\tau^{(n)}\sim \CJop$ and $\tau^{(\infty)}\sim  \Huaop$.

By Propositions \ref{prop:y_cj} and \ref{prop:xtight} there exists a tight sequence $\kappa^{(n)}, n\in \ZZ_+$ so that the inequalities \eqref{eq:xyfluc_bound} and \eqref{eq:xyfluc_bound_x} are satisfied for $n\in \ZZ_+$ with $c_1=c_\delta-\eps, c_2=c_\delta+\eps$, $c_3=c_\delta-\eps$. Here $c_\delta=\frac{4}{\beta}(\Re \delta+1/2)$ and $\eps\in (0, \min(c_\delta, \frac12))$ is arbitrary. By \eqref{ineq:hpsde} there is a finite random variable $\kappa^{(\infty)}$ so that 
\eqref{eq:xyfluc_bound} and \eqref{eq:xyfluc_bound_x} are satisfied for $\tl x+i \tl y$ with the just defined $c_1, c_2, c_3$. Together with Proposition \ref{prop:cjpath} this means that the conditions of Corollary \ref{cor:randomDirac} are satisfied, and hence the statement of the theorem follows. 
\end{proof}

The proof of Theorem \ref{thm:RO_op_limit} follows along the same line.

\begin{proposition}\label{prop:y_orth}
Fix $\beta>0$, $a,b>-1$. Let $y^{(2n)}_k, 0\le k\le 2n$ be  defined as in Proposition \ref{prop:orthpath}. Then for any $\eps>0$ small, there exists a sequence of tight random variables $\kappa^{(2n)}=\kappa^{(2n)}(\eps)$ such that for all $0\le k\le 2n-1$,
\[
(\kappa^{(2n)})^{-1}(1-\tfrac{k}{2n})^{2a+1+\eps}\le y^{(2n)}_{k}\le \kappa^{(2n)}(1-\tfrac{k}{2n})^{2a+1-\eps}.
\]
\end{proposition}
\begin{proof}
One can just mimic the steps of the  proof of Proposition \ref{prop:y_cj} using the parameters 
\begin{align*}
(s^{(2n)}_{k},t^{(2n)}_{k}) =\begin{cases}
\big(\tfrac{\beta}{4}(2n-k+2a),\tfrac{\beta}{4}(2n-k+2b)\big) &\mbox{if $k$ is even,}\\
\big(\tfrac{\beta}{4}(2n-k+2a+2b+1), \tfrac{\beta}{4}(2n-k-1)\big) &\mbox{if $k$ is odd},
\end{cases}
\end{align*}
 and $c_1=2a+1$, $c_2=\frac{4}{\beta}$. 
\end{proof}

\begin{proof}[Proof of Theorem \ref{thm:RO_op_limit}] 
    Consider the random variables $ y^{(2n)}_{k}, 0\le k\le n$ defined in Proposition \ref{prop:orthpath}, and define $(x^{(2n)}(t), y^{(2n)}(t)):=(0,y^{(2n)}_{\lfloor 2nt\rfloor})$. Let $\tl y$ be the process  defined  in Proposition \ref{prop:hardedge} and set $\tl x=0$. 
Set $q^{(2n)}=q=0$, and define $\tau^{(2n)}, n\in \ZZ_+$ using $(x^{(2n)}+i y^{(2n)}, q^{(2n)})$, and $\tau^{(\infty)}$ using  $(\tl x+i \tl y,q)$. Then $\tau^{(2n)}\sim \ROop$ and $\tau^{(\infty)}\sim  \Bessop$.

By Propositions \ref{prop:y_orth}  there exists a tight sequence $\kappa^{(2n)}, n\in \ZZ_+$ so that the  inequalities \eqref{eq:xyfluc_bound} and \eqref{eq:xyfluc_bound_x} are satisfied for $n\in \ZZ_+$ with $c_1=2a+1-\eps, c_2=2a+1+\eps$, $c_3=\max(c_1, 1)$. (Note that since $x^{(2n)}=q^{(2n)}=0$ the inequality \eqref{eq:xyfluc_bound_x} holds for  any positive $c_3$.) Here  $\eps\in(0,\frac12)$ is chosen so that $c_1>-1$.  By the sublinearity of  Brownian motion there is a finite random variable $\kappa^{(\infty)}$ so that 
\eqref{eq:xyfluc_bound} and \eqref{eq:xyfluc_bound_x} are satisfied for $\tl x+i \tl y$ with the just defined $c_1, c_2, c_3$. Together with Proposition \ref{prop:orthpath} this means that the conditions of Corollary \ref{cor:randomDirac} are satisfied, and hence the statement of the theorem follows.
\end{proof}

\section{Proofs of the theorems related to the limiting operators}\label{s:limitingobjects}

In this section we provide the proofs 
%of Theorems \ref{thm:KSsde}, \ref{thm:HP_2piZ}, \ref{thm:zetaHP} and \ref{thm:zetaB}, these are 
for
our results on the properties and characterizations of the limiting point processes and random analytic functions arising from the {circular Jacobi $\beta$-ensemble} and the  {real orthogonal $\beta$-ensemble} (Theorems \ref{thm:KSsde}, \ref{thm:HP_2piZ}, \ref{thm:zetaHP} and \ref{thm:zetaB}).

\subsection{Simple transformations of Dirac operators}\label{subs:reverse}

For some of our results it will be more convenient to consider Dirac operators that live on $(0,1]$, with a potential limit point at $0$. (In  fact this is the framework used in  \cite{BVBV_szeta}.) In order to do this, the framework introduced in Section \ref{subs:Diracop} has to be extended to also include the following setup (we call this the \emph{reversed framework}):
\\
a) Both the generating path $x+i y$ and the weight function $R$ (defined via \eqref{eq:R}) are defined on $(0,1]$.  The operator $\tau$ in \eqref{def:Dirop} acts on $(0,1]\to \R^2$ functions.\\
b)  In Assumption \ref{assumption:1} the first integral condition is replaced with $\int_0^1 \| R(s) \uu_0\| ds<\infty$.

Otherwise we have the same assumptions: $x+i y$ is measurable and locally bounded on its domain, the boundary conditions $\uu_0, \uu_1$ satisfy \eqref{eq:u_assumption}. Then $\tau$ is self-adjoint on the domain $\dom (\tau)$ given by \eqref{Dir:domain}, its inverse is a Hilbert-Schmidt integral operator with the kernel given in \eqref{eq:HS_kernel}. The operator $\res \tau$, the integral trace $\mathfrak{t}_\tau$, and the secular function $\zeta_\tau$ can be defined the same way as before (see Section \ref{subs:Diracop}).

% We call this the reversed framework, as we soon see it can be connected with our usual 

% We start by introducing the reversed framework and other transformations of Dirac operators. 

% We consider the reversed framework where the Dirac operators are defined on $(0,1]$. 

% The endpoint 0 could possibly be limit point, in fact this is the framework used in  \cite{BVBV_szeta}. 

% Let $R$ be the weight function defined by \eqref{eq:R} with a locally bounded measurable function $x+iy:(0,1]\to\HH$, and $\uu_0, \uu_1$ be nonzero, non-parallel vectors in $\RR^2$. We assume that $R(t)$ is positive definite for all $t\in (0, 1]$, and both $\|R\|$ and $\|R^{-1}\|$ are locally bounded on $(0, 1]$. We also assume that $\uu_0, \uu_1\in \R^2$ satisfy \eqref{eq:u_assumption}, and the following integrability condition. 

% \begin{assumption} \label{a:reverse}
% \begin{align*}
% \int_0^1 \|R \uu_0\|ds<\infty, \quad \quad     \int_0^{1} \int_0^t \mathfrak  u_0  ^t R(s) \mathfrak  u_0 \, \mathfrak (\uu_1)^t R(t) \mathfrak  \uu_1 ds dt<\infty.
% \end{align*}
% \end{assumption}
% Then the operator $\tau$ defined via \eqref{def:Dirop} on the domain $\dom(\tau)$ of \eqref{Dir:domain} is self-adjoint, and its inverse is Hilbert-Schmidt with the integral kernel given in \eqref{eq:HS_kernel}. The operator $\res \tau$, the integral trace $\mathfrak{t}_\tau$, and the secular function $\zeta_\tau$ can be defined the same way as before (see Section \ref{subs:Diracop}).

There is a simple way to move between the two frameworks. Introduce the time reversal  operator $\rho f(t):=f(1-t)$ acting on functions defined on $[0,1)$ or $(0,1]$. Let $\iota: \HH\to \HH$ be defined as the reflection $x+i y\to -x +i y$, and set
\[
S=\mat{1}{0}{0}{-1}.
\]
If a weight function $R$ is generated by the path $z=x+i y$, then $SRS$ is the weight function corresponding to the path $\iota z$.

The statements of the following two lemmas are contained in Lemma 36 of \cite{BVBV_szeta}.
\begin{lemma}[\cite{BVBV_szeta}]\label{lem:timerev}
	Assume that the Dirac operator $\tau=\Dirop(R,\uu_0, \uu_1)$ satisfies the assumptions \eqref{eq:u_assumption} and \eqref{eq:int_assumption} with boundary conditions $\uu_0,\uu_1$, weight function $R$, and generating path $z=x+i y$.  Then the operator $\rho^{-1}S \tau S \rho$ satisfies the assumptions of the reversed framework with boundary conditions $-\uu_1, -\uu_0$, weight function $\rho SRS$, and generating path $\iota \rho z$.
	The operators $\tau$ and $\rho^{-1}S \tau S \rho$ are orthogonally equivalent  in the respective $L^2$ spaces,  they have the same integral traces and secular functions.
\end{lemma}

\begin{lemma}[\cite{BVBV_szeta}]\label{lem:rotation}
	Let $Q$ be a $2\times 2$ orthogonal matrix with determinant 1. Let $\mathcal Q: \bar \HH\to \bar \HH$ be the corresponding linear isometry of $\bar \HH$ mapping $z\in \bar \HH$ to the ratio of the entries of $Q [z,1]^t$. Suppose that the Dirac operator $\tau$ satisfies the assumptions \eqref{eq:u_assumption} and \eqref{eq:int_assumption} with boundary conditions $\uu_0, \uu_1$ and generating path $x+i y$. Then the operator $Q \tau Q^{-1}$ also satisfies the same assumptions, with boundary conditions $\mathcal Q \uu_0, \mathcal Q \uu_1$ and generating path $\mathcal Q(x+i y)$. The two operators are orthogonally equivalent,  they have the same integral traces and secular functions. The same statement holds if $\tau$ satisfies the assumptions of the reversed framework.
\end{lemma}

\subsection{Proofs of the theorems related to $\Huaop$}

Our first step is to produce a unitary equivalent form of the operator $\Huaop$ where the driving path is independent of the boundary conditions. In order to do that, we use the following factorization lemma for the diffusion \eqref{eq:HPsde}. This is a generalization of Proposition X.3.1 in \cite{hypBM} which treats the $\delta=0$ case, i.e.~the hyperbolic Brownian motion. 

We recall that in the Poincar\'e  half plane model of the hyperbolic plane the isometries are of the form $z\to \frac{a z+b}{c z+ d}$ with $a, b, c, d\in \R$ and $ad-bc\neq 0$. For $r\in \R$ we set
\begin{align}\label{def:hyprot}
    T_r(z)=\frac{rz+1}{r-z}.
\end{align}
$T_r$ is  the hyperbolic rotation  about the point $i$ taking $r$ to $\infty$ and $\infty$ to $-r$.

\begin{theorem}\label{thm:factor} Fix $\delta\in \CC$ with $\Re \delta>-1/2$. 
 Consider the diffusion $w=x+i y$ defined in \eqref{eq:HPsde}, and denote by $w_\infty$ the a.s.~limit as $t\to \infty$. Then the process $\tl w_t=T_{w_\infty} w_t$ satisfies the diffusion
 \begin{align}\label{eq:SDE_cond}
     d\tl w=\Im \tl w(d\tl Z+i (1+\bar \delta) dt), \qquad \tl w_0=i.
 \end{align}
 where $\tl Z$ is standard complex Brownian motion. 

Moreover, if a process $\tl w$ satisfies the SDE \eqref{eq:SDE_cond}, and $q$ is a random variable with distribution $P_{IV}(\Re \delta+1,-2 \Im \delta)$ then the process $x_t+i y_t=T_q^{-1} \tl w_t$ satisfies the SDE \eqref{eq:HPsde} with $B_1, B_2$ being independent copies of standard Brownian motion. 
\end{theorem}
\begin{proof}
By Theorem \ref{thm:q} the distribution of $w_\infty$ is given by $P_{IV}(\Re \delta+1,-2 \Im \delta)$.
 The SDE \eqref{eq:HPsde} is invariant under affine transformations of the form $z\to a+b z$ with $a\in \R, b>0$.  Hence for $a\in \R, b>0$ the solution of \eqref{eq:HPsde} with initial condition $a+i b$ will converge in distribution to $a+b w_\infty$ where $w_\infty\sim P_{IV}(\Re \delta+1,-2 \Im \delta)$. Now using either Doob's $h$-transform or the technique of enlargement of filtrations (c.f.~\cite{RogersWilliams}, or \cite{MatsumotoYor2001}) one can  show that for a given $r\in \RR\cup \{\infty\}$ the process $w$ conditioned on the event $\{w_\infty=r\}$ satisfies the diffusion
\begin{align}
  dz^{(r)}=\Im z^{(r)} \left(dZ+
i (1+\bar \delta)\frac{z^{(r)}-r}{\overline{z^{(r)}}-r} dt\right), \qquad z^{(r)}(0)=i.  
\end{align}
Here $Z$ is a standard complex Brownian motion, and in the $r=\infty$ case the $\frac{z^{(r)}-r}{\bar z^{(r)}-r}$ term in the drift is replaced by the constant one. In particular, $z^{(\infty)}$ has the same distribution as the process $\tl w$ from \eqref{eq:SDE_cond}, and it hits $\infty$ with probability one.
Using Ito's formula one can readily check that  for $r\in \R$ the rotated process $\tl w^{(r)}=T_r(z^{(r)})=\frac{rz^{(r)}+1}{r-z^{(r)}}$ satisfies the SDE     \eqref{eq:SDE_cond}, in particular, its distribution does not depend on $r$. 
This shows that the rotated process $t\to  T_{w_\infty} w_t$ has the same distribution as $\tl w$ from \eqref{eq:SDE_cond}, and that it is independent of  $w_\infty$. Using $w_\infty\sim P_{IV}(\Re \delta+1,-2 \Im \delta)$ the second half of the theorem follows as well.
% Note that 
% \[
% q-w=\frac{q^2+1}{q+\tl w}, \qquad  \Im w = \frac{q^2+1}{|q+\tl w|^2}\Im \tl w.
% \]
% By It\^o's formula we have
% \[
% d\tilde w=\frac{(q+\tl w)^2}{q^2+1} dw = \Im\tl w\frac{q+\tl w}{q+\bar{\tl w}}dZ + i (\bar \delta+1)\Im\tl wdt.
% \]
% The noise term is 
% \[
% \Im \tilde w \frac{q+\tl w}{q+\bar{\tl w}}dZ=\Im \tilde w d\tilde Z
% \]
% with a new standard complex Brownian motion $\tilde Z$. This proves \eqref{eq:SDE_cond}.
% The SDE \eqref{eq:SDE_cond} can be solved explicitly, with $\tilde w=\tilde x+i \tilde y$, and $\tilde Z=\tilde B_1+i \tilde B_2$ we get
% \[
% \tilde y(t)=e^{\mathrm{\tl B}_2(t)+(\Re\delta+\frac12)t}, \qquad \tilde x(t)=\Im\delta\int_0^t\tl y(s) ds + \int_0^t\tl y(s) d\mathrm{\tl B}_1(s).
% \]
% Since $\Re \delta+1/2>0$, the  diffusion $\tl w$ hits $\infty$ with probability one. This proves the factorization property for the diffusion $w$, i.e., the process $w$ can be factorized into two independent pieces, the boundary point and the same process conditioned to hit $\infty$. On the other hand, by rotating the process $\tl w$ so that the boundary point is distributed according to \eqref{eq:HP_q_pdf} one can prove that $T_q^{-1}\tl w$ satisfies the SDE \eqref{eq:HPsde}. This completes the proof.
\end{proof}

We will now construct a reversed and transformed version of $\Huaop$. Let $B_1,B_2$ be independent two-sided real Brownian motion. Consider the two-sided version of $x+iy$  from \eqref{eq:HPsde} defined using $B_1,B_2$, i.e., 
\begin{align}\label{eq:HPtwosided}
y_s= e^{B_2(s)-(\Re\delta+\frac12)s}, \quad 
x_s=\begin{cases}
-\int_s^0 y(t) dB_1 - \Im\delta\int_s^0 y(t)dt \quad&\mbox{$s\leq 0$},\\
\int_0^s y(t) dB_1  + \Im\delta\int_0^s y(t)dt \quad&\mbox{$s\geq 0$}.
\end{cases}
\end{align} 
We also introduce the time change 
\begin{equation*}
u_\beta(t)=-\tch_\beta(1-t)=\frac{4}{\beta}\log t.
\end{equation*}

\begin{definition}\label{def:HP_rev}
Let $q$ be a random variable with distribution $ P_{IV}(1+\Re\delta,-2\Im\delta)$ independent of $B_1,B_2$. Set $\hat x(t)+i\hat y(t)=x(u_\beta(t))+iy(u_\beta(t))$ for $t\in(0,1]$. Define the reversed and transformed version of 
the $\Huaop$ operator as 
\begin{equation*}
 \tau^{\ttHP}_{\beta,\delta}=\Dirop(\hat x+i \hat y, \mathfrak{u}_0,\mathfrak{u}_1), 
\end{equation*}
where $\uu_0=[1,0]^t, \uu_1=[-q,-1]^t$. 
\end{definition}

In this section we will use the simplified notation $\tau_{\beta, \delta}$ for $ \tau^{\ttHP}_{\beta,\delta}$, and denote 
the secular function of $\tau_{\beta, \delta}$ by $\zeta_{\beta, \delta}$.

\begin{lemma}\label{lem:HPtau}
%Fix $\beta>0$ and $\delta \in \CC$ with $\Re \delta>-1/2$. Let $B_1, B_2$ be independent two-sided real Brownian motion, and define $x+i y$ defined as in \eqref{eq:HPtwosided}. 
%Let $q$ be a random variable with distribution $P_{IV}(\Re \delta+1,-2 \Im \delta)$ independent of $B_1,B_2$.
%
%Set $\mathfrak{u}_0=[1,0]^t, \mathfrak{u}_1=[-q,-1]^t$. For $t\in (0,1]$ let  $\hat x(t)+i \hat y(t)=x(u_\beta(t))+i y(u_\beta(t))$, and set $\tau_{\beta,\delta}=\Dirop(\hat x+i \hat y, \mathfrak{u}_0,\mathfrak{u}_1)$.
The operator $\tau_{\beta,\delta}$ is orthogonal equivalent to an operator which has the same distribution as the $\Huaop$ operator. In particular,  the random analytic function $\zeta_{\beta,\delta}$  has the same distribution as $\zeta^{\ttHP}_{\beta, \delta}$. 
\end{lemma}

\begin{proof}
Recall the  transformations $\iota, S$ and $\rho$  defined in and around Lemma \ref{lem:timerev}. Let $T_q$ be the hyperbolic rotation defined in \eqref{def:hyprot}. Consider the Dirac operator 
\begin{align*}
\tl\tau =\rho^{-1}S\,\Dirop(T_q(\hat x +i \hat y),T_q\uu_0,T_q\uu_1)\,S\rho = \Dirop(\rho\iota T_q(\hat x+i \hat y),-T_q\uu_1,-T_q\uu_0).
\end{align*} 
 Here we identify the boundary condition $\uu=[a,b]^t$ with its projection $a/b$ onto the real axis so that $T_q \uu_0$, $T_q \uu_1$ are well defined:
 \begin{align*}
    -T_q\uu_1=\infty, \qquad -T_q\uu_0=q.
 \end{align*}
 By Lemmas \ref{lem:timerev} and \ref{lem:rotation} the operator $\tl \tau$ is orthogonal equivalent to $\tau_{\beta,\delta}$, hence we just have to  show that $\tl\tau$ has the same distribution as $\Huaop$.

Note that $T_q=T_{-q}^{-1}$ and $-q\sim P_{IV}(\Re\delta+1,2\Im\delta)$. 
From the definition \eqref{eq:HPtwosided} it follows that the reversed process $(x_{-s}+iy_{-s}), s\ge 0$ satisfies the SDE \eqref{eq:SDE_cond} with drift $i(1+\delta)$ in place of $i(1+\bar\delta)$. Hence by Theorem \ref{thm:factor}, the process $T_q(x_{-s}+iy_{-s}), s\ge 0$ satisfies the SDE 
\begin{align*}
dw = \Im w(dZ-i\bar{\delta}ds),\qquad w(0)=i,
\end{align*}
with  standard complex Brownian motion $Z$, and the path converges to $T_q \infty=-q$ as $s\to\infty$. From this it follows that
\begin{align*}
	\rho\iota T_q(x_{u(\cdot)}+iy_{u(\cdot)}) \ed \rho (x_{-u(\cdot)}+iy_{-u(\cdot)})= (x_{\upsilon_\beta(\cdot)}+iy_{\upsilon_\beta(\cdot)}),
\end{align*}
with $\lim_{t\to 1}\rho\iota T_q(x_{u(t)}+iy_{u(t)})=q$.
 This shows that the driving path and boundary conditions of $\tl \tau$ match up (in distribution)  with the corresponding ingredients of the $\Huaop$ operator, proving the statement of the lemma. 
\end{proof}

The independence of the boundary point and the driving path in the reversed operator  $\tau_{\beta,\delta}$ allows us to prove Theorem \ref{thm:zetaHP}.
Our proof follows the proof of Theorem 1 of \cite{BVBV_szeta}, which can be considered the $\delta=0$ case of our theorem.

\begin{proof}[Proof of Theorem \ref{thm:zetaHP}]
	By Lemma \ref{lem:HPtau} the random analytic function $\zeta_{\beta,\delta}$ has the same distribution as $\zeta^{\ttHP}_{\beta, \delta}$. Hence  we can work with the reversed operator $\tau_{\beta,\delta}$, and prove the statements of the theorem for $\zeta_{\beta,\delta}$.
	
	By Proposition 13 in \cite{BVBV_szeta} the secular function of $\tau_{\beta,\delta}$ can be characterized as follows. 
	Let $R(t)$ be the weight function built from the driving path of the reversed $\tau_{\beta,\delta}$ operator according to \eqref{eq:R}. Then there exists a unique function $H:(0,1]\times \CC\mapsto\CC^2$ so that for every $z\in\CC$ the function $H(\cdot,z)$ solves the ODE
	\begin{align}\label{eq:HP_H}
	J\frac{d}{dt}H(t,z) = zR(t)H(t,z),\qquad \lim_{t\to 0}H(t,z) = \uu_0=[1,0]^t.
	\end{align}
	The secular function $\zeta_{\beta,\delta}$ can be obtained from $H$ using the formula $\zeta_{\beta,\delta}(z)=[1,-q]H(1,z)$.
	
	Consider the process $X_u = \begin{pmatrix}1 &-x_u\\0&y_u\end{pmatrix},u\le 0$, where $x_u+iy_u$ is defined in \eqref{eq:HPtwosided}. Define $\cH_u(z)=X_uH(t(u),z)$ with $t(u)=e^{\frac{\beta}{4}u}$ being the inverse of $u(t)=\tfrac{4}{\beta}\log t$. 
	Since $X_0=\mat{1}{0}{0}{1}$, we have  $\zeta_{\beta,\delta}(z)=[1,-q]\cH_0(z)$.  A direct computation using  It\^o's formula  shows that $\cH_u$ solves the SDE \eqref{eq:HP_cH}. To be precise, one first has to consider approximations  of $\cH_u$ that are defined on $[\eps, 1]$, for this one has to use the approximation method introduced in Propositions 20 and 43 in \cite{BVBV_szeta}. A simple extension of those arguments also shows the characterization of $\cH_u(z)$ as the unique solution of \eqref{eq:HP_cH} with the conditions given. 
	
	Now write $\cH_u=[\cA_u,\cB_u]^t$. The functions $\cA_u, \cB_u$ are entire functions on $\CC$, we denote their Taylor coefficients at 0 by $\cA^{(n)}_{u}, \cB^{(n)}_{u}$.
	Since the SDE system \eqref{eq:HP_cH} depends analytically on its parameter $z$, It\^o's formula can be applied to get SDEs for derivatives in this parameter as well, see e.g.~Section V.7 of \cite{Protter}. 
	Differentiating \eqref{eq:HP_cH} $n$ times in $z$ and considering $z=0$ shows that the Taylor coefficients $\cA^{(n)},\cB^{(n)}$ satisfy the following system of SDEs
	\begin{align*}
	d\cB^{(n)} &= \cB^{(n)} dB_2 -\Re\delta \cB^{(n)} du -\frac{\beta}{8}e^{\beta u/4}\cA^{(n-1)}du,\\
	d\cA^{(n)} &= -\cB^{(n)} dB_1 -\Im\delta \cB^{(n)} du +\frac{\beta}{8}e^{\beta u/4}\cB^{(n-1)}du,
	\end{align*} 
	with initial conditions $\cB^{(0)}\equiv 0$, $\cA^{(0)}\equiv 1$. Mimicking the proof of Propositions 45 and 47 in \cite{BVBV_szeta} one can prove that the solution of the above system exist, and it is given by equations \eqref{eq:HP_taylor}, \eqref{eq:HP_taylor_2}.
\end{proof}

Using the SDE characterization of $\zeta^{\ttHP}_{\beta, \delta}$ given in Theorem \ref{thm:zetaHP} we are able to prove Theorem \ref{thm:KSsde}.

\begin{proof}[Proof of Theorem \ref{thm:KSsde}]
As in the proof of Theorem \ref{thm:zetaHP}, we work with the operator $\tau_{\beta,\delta}$. The spectrum of this operator has the same distribution as the $\HPb$ process.

Consider the random analytic function valued processes $\cA_u, \cB_u$ introduced in the proof of Theorem \ref{thm:zetaHP}. Recall
 that $\zeta_{\beta,\delta}=[1,-q]\cH_0=\cA_0-q\cB_0$, with $q$ given in the definition of $\tau_{\beta, \delta}$, see Definition \ref{def:HP_rev}. 
 
 We introduce the structure function $\cE(u,z)=\cA_u(z)-i\cB_u(z)$, note that this can also be expressed as $[1,-i]\cH(u,z)$ with $\cH_u$ defined in the proof of Theorem \ref{thm:zetaHP}.
 For $\lambda\in\R$ we define $2\log\cE(u,\lambda)=\mathcal{L}_\lambda(u)+i\alpha_\lambda(u)$ with $\mathcal{L}_\lambda, \alpha_\lambda \in \R$, where for each $u\in \R$ the function is chosen so that it is continuous in $\lambda$ and $\alpha_0(u)=0$. (This is possible because $\cH_u(z)$ is continuous in $z$ and it is never equal to $[0,0]^t$.)
By \eqref{eq:HP_cH} and It\^o's formula we get
\begin{align}\label{eq:alphaSDE_2}
	d\alpha_\lambda = \lambda\tfrac{\beta}{4}e^{\frac{\beta}{4}u} du + \Re[(e^{-i\alpha_\lambda}-1)(dZ-i\delta du)], \qquad \alpha_\lambda(-\infty)=0. 
\end{align}
The process $\psi_\lambda(t)=\alpha_\lambda(u(t))$ with $u(t)=\tfrac4{\beta} \log t$ satisfies the SDE \eqref{eq:HP_psi_sde}, and simple coupling arguments show that it is the unique solution of \eqref{eq:HP_psi_sde} with the conditions given in Theorem \ref{thm:KSsde}. (See e.g~\cite{KS} for more details in the $\delta=0$ case.)

Set $\theta = -2\arccot q$. By the comment following Theorem \ref{thm:q}  we have  $e^{i \theta}\sim \Theta(1,\delta)$, and $\theta$ is independent of the complex Brownian motion $Z$ in \eqref{eq:alphaSDE_2}.
The eigenvalues of $\tau_{\beta,\delta}$ are given by the zeros of $\zeta_{\beta, \delta}$. By definition we have $\zeta_{\beta, \delta}(\lambda)=0$ if and only if $\cE(0,\lambda)$ is a real multiple of $q-i$, or equivalently $\alpha_\lambda(0)=\psi_{\lambda}(1)=2\log (q-i)=\theta$ mod $2\pi$. Using $\spec(\tau_{\beta,\delta})\ed\HPb$ finishes the proof.
\end{proof}

Now we turn to the proof of Theorem \ref{thm:HP_2piZ}. 
We first isolate the statements regarding the SDE \eqref{eq:HP_alpha_sde} in a separate lemma. 

\begin{lemma}\label{lem:HP_alpha_sde}
The SDE system \eqref{eq:HP_alpha_sde} has a  unique strong solution on $t\in [0,\infty)$, $\lambda\in \R$.
With probability one the process $\lambda\to \alpha_\lambda(t)$ is increasing for all $t>0$. 
For each $\lambda\in \R$ the limit $\lim\limits_{t\to \infty}\frac{1}{2\pi} \alpha_\lambda(t)$ exists almost surely and it is an integer.  
Moreover, if $\beta\le 4(\Re\delta+\frac12)$ and $\lambda>0$ then a.s.~$\frac{1}{2\pi} \alpha_\lambda(t)$ converges to an integer from above. 
\end{lemma}
Note that for $\delta=0$ these statements were proved in Theorem 7 and Proposition 9 of \cite{BVBV}.

\begin{proof}
The fact that   the  system \eqref{eq:HP_alpha_sde} has a  unique strong solution follows from  standard theory, the monotonicity property is a consequence of the monotone dependence of the drift function  of the parameter $\lambda$. 

For a fixed $\lambda\in \R$ the process $\alpha_\lambda$ solves the SDE  
	\begin{align}\label{eq:SDE_1234}
	d\alpha_{\lambda} = \lambda \tfrac{\beta}{4}e^{-\frac{\beta}{4}t}dt + (\Im\delta(\cos\alpha_{\lambda}-1)-\Re\delta\sin\alpha_{\lambda})dt + 2\sin(\tfrac{\alpha_{\lambda}}{2})dW, \quad \alpha_{\lambda}(0)=0,
	\end{align}
 where $W$ is a standard real Brownian motion depending on $\lambda$.    
    
	For $\lambda=0$ we  have $\alpha_\lambda(t)=0$. It is sufficient to show the statement for $\lambda>0$, since $-\alpha_{-\lambda}$ solves the same SDE as $\alpha_\lambda$ with $\bar \delta$.  From the monotonicity in $\lambda$ it follows that for $\lambda>0$ we have  $\alpha_\lambda(t)>0$ for $t>0$ almost surely, and if $t_0>0$, $m\in \ZZ$ then on the event $\alpha_\lambda(t_0)>2 m \pi$ one has $\alpha_\lambda(t)>2m \pi$ for all $t>t_0$ with probability one. (See Proposition 9 in \cite{BVBV} for the proof of these statements in the $\delta=0$ case.)
	
	Fix $\lambda>0$, and introduce the  diffusion
	\begin{align*}
	 X(t) =   \begin{cases}
	 \log(\tan(\alpha_\lambda(t)/4)), \qquad &\text{if } \alpha_\lambda(t)\in[4k\pi, (4k+2)\pi),\\
	 -\log(-\tan(\alpha_\lambda(t)/4)), \qquad &\text{if }  \alpha_\lambda(t)\in[(4k+2)\pi, (4k+4)\pi).
	\end{cases} 
	\end{align*}
	By It\^o's formula, this diffusion  satisfies the SDE
	\begin{align}\label{eq:Xsde}
	dX = \tfrac{\lambda\beta}{8}e^{-\beta t/4}\cosh Xdt + (\Re\delta+\tfrac12) \tanh X_i dt -\Im\delta\sech Xdt+ dW,\,\, X(0)=-\infty,
	\end{align}
	with a $W$ standard Brownian motion that is a simple transformation of the $W$ from \eqref{eq:SDE_1234}. Note that the diffusion might blow up to $\infty$ in finite time, in which case it restarts immediately from $-\infty$. To prove the convergence statement for $\tfrac{1}{2\pi} \alpha_\lambda(t)$ we need to show that with probability one $\lim\limits_{t\to \infty} X(t)$ exists and it is an element of $\{-\infty,\infty\}$. 
	 This can be proved with fairly straightforward coupling arguments, we will only give a sketch of the proof.
	
	For given $t_0>0, x\in \RR$ we can consider the solution of \eqref{eq:Xsde} on $[t_0,\infty)$ with $X(t_0)=x$. We denote the distribution of the process by $P_{t_0, x}$.

	Denote the drift term in the SDE \eqref{eq:Xsde} by 
	\begin{equation*}
	R(x,t) =\tfrac{\lambda\beta}{8}e^{-\beta t/4}\cosh x+(\Re\delta+\tfrac12)\tanh x - \Im\delta\sech x.
	\end{equation*}
	Note that when $|x|\leq 2M$, the function $|R(x,t)|$ could be  bounded from above by a constant $c=c(M,\delta,\beta,\lambda)$ that is independent of  $t$. By coupling $R$  with a Brownian motion with drift  $c$, it follows that for any fixed $M>0$ there is an $\eps\in (0,1)$ so that 
	\[
	\sup_{t_0>0, |x|\le M} P_{t_0, x}\left(|X(t)|\le M \text{ for all } t\in [t_0,t_0+1]\right)\le 1-\eps.
	\]
	Using the strong Markov property it now  follows that for any $t_0>0$, $x\in [-M,M]$ we have 
	\begin{align}\label{eq:X_111}
	    	P_{t_0,x}\left(|X(t)|\le M \text{ for all } t\ge t_0\right)=0.
	\end{align}
We will show that there is a positive  constant $c_1$, so that 
\begin{align}\label{eq:X_222}
\lim_{M\to \infty}\, \inf_{\substack{t_0\ge c_1 M\\ |x|\ge M}} P_{t_0, x}(\lim_{t\to \infty} X(t)\in \{-\infty,\infty\})=1.
\end{align}
This statement together with \eqref{eq:X_111} implies that with probability one $\lim_{t\to \infty} X(t)\in \{-\infty, \infty\}$.

Fix $x\ge M$, $t_0>0$. 	 For any fixed $0<c_+<\Re\delta+\frac12$, we could choose $M$ large so that $R(x,t)\geq  c_+$ for all $x\geq M/2, t\ge 0$. Under the  distribution $P_{t_0,x}$, the coupling \[
 X(t)-M\geq W_{c_+}(t_0,t):= W(t)-W(t_0)+c_+(t-t_0)
 \]
 holds on $[t_0,\sigma]$ where
	\begin{align*}
		\sigma:=\inf_{t\ge t_0}\{X(t-)=\infty \text{ or } W_{c_+}(t_0,t)\le -M/2\}.
	\end{align*}
	Since $c_+>0$, the random variable $-\inf_{t\geq t_0} W_{c_+}(t_0,t)$ is distributed as an exponential random variable with parameter $2c_+$ (see e.g.~\cite{MY_PS2}). Thus,
	\begin{align*}
	P_{t_0,x}(W_{c_+}(t_0,t)> -\tfrac{M}{2},\forall  t\ge t_0) =1-e^{-c_+M}.
	\end{align*}
	Using the sublinearity  of Brownian motion we get that
	\begin{align}\label{eq:lim_Xinfty}
	\inf_{\substack{t_0> 0\\ x\ge M}} P_{t_0,x}(\lim_{t\to\infty} X(t)=\infty \text{ or } X(t)\text{ blows up in finite time})\ge 1-e^{-c_+M}.
	\end{align}
	
	Next we fix the constants $c_{-}, c_2$ with $0<c_{-}<c_2<\min\{\Re\delta+\frac12,\frac{\beta}{4}\}$, and fix $t_0\ge 2c_2^{-1} M$, $x_0\le -M$. The bound $R(x,t)\le -c_{-}$ holds in the region 
	\[
	\mathcal{R}:=\{(t,x):-M/2\ge x\ge -c_2 t, t\ge t_0\},
	\]
	if $M$ is larger than a fixed constant that only depends on $\lambda, \delta$ and $\beta$. 
	Thus under $P_{t_0,x_0}$ we can couple $X(t)-x_0$ on $[t_0,\infty)$ from above with the process \[
	W_{-c_{-}}(t_0,t):=W(t)-W(t_0)-c_{-}(t-t_0),\]
	on the event that $(t,-M+W_{-c_{-}}(t_0,t))$ stays in the region $\mathcal{R}$. Note that by our assumption $(t_0,-M+W_{-c_{-}}(t_0,t_0))\in \mathcal{R}$. Note that both
	\[\sup_{t\ge t_0} W_{-c_{-}}(t_0,t)\quad \text{and} \quad -\inf_{t\ge t_0} W_{-c_{-}}(t_0,t)+c_2(t-t_0)
	\]
	are exponentially distributed, with parameters $2c_{-}$ and $2(c_2-c_-)$, respectively. Hence the probability of $(t,-M+W_{-c_{-}}(t_0,t))$ not staying in the region $\mathcal{R}$ is exponentially small in $M$. Since $-M+W_{-c_{-}}(t_0,t)$ converges to $-\infty$ as $t\to \infty$, we get 
	\begin{align}\label{eq:lim_Xneginfty}
	\lim_{M\to \infty} \inf_{\substack{t_0> 2 c_2^{-1} M\\ x_0\le -M}} P_{t_0,x_0}(\lim_{t\to\infty} X(t)=-\infty )=1.
	\end{align}
From \eqref{eq:lim_Xinfty} and \eqref{eq:lim_Xneginfty} we get \eqref{eq:X_222}, which implies that a.s.~$X$ converges to either $\infty$ or $-\infty$.
		
	In the case $\beta\le 4(\Re\delta+\frac12)$, the $\Huaop$ operator is limit point at $t=1$. In this case, for $\lambda>0$ one can show that the limit of $X(t)$ has to be $-\infty$. This generalizes Theorem 7 of \cite{BVBV} which proves the statement for $\delta=0$. 
	The idea is that for any fixed $\delta$ with $\Re \delta+1/2>0$ one can choose $M$ large so that the term $-\Im\delta\sech x$ in $R(x,t)$ is negligible on the event $\{X(t)\ge M \text{ for } t\ge t_0\}$. After dropping that term, one can just mimic  the proof of the $\delta=0$ case from Theorem 7 of \cite{BVBV}.
	 This proves that a.s.~$X$ converges to $-\infty$ when $\beta\le 4(\Re\delta+\frac12)$ and hence a.s.~$\alpha_\lambda$ converges from above for any fixed $\lambda>0$.
\end{proof}

We now have all the ingredients to prove Theorem \ref{thm:HP_2piZ}. 

\begin{proof}[Proof of Theorem \ref{thm:HP_2piZ}]
The statements about the SDE \eqref{eq:HP_alpha_sde} are proved in Lemma \ref{lem:HP_alpha_sde}. The rest of the proof will follow along the lines of the proof of Theorem 26 in \cite{BVBV_op}, where the $\delta=0$ case is handled.

Consider the operator $\Huaop$ defined in Proposition \ref{prop:hpoperator}. Let $v=v_\lambda=[v_1,v_2]^t$ be the solution of the differential equation $\Huaop v=\lambda v$ with $v(0)=[1,0]^t$. Then the ratio of the two components $r_{\lambda}(t)=\frac{v_1(\lambda,t)}{v_2(\lambda,t)}$ satisfies the  ODE
	\begin{align}\label{eq:r_ODE}
	r_\lambda'=\lambda\tfrac{\tl y^2+(\tl x-r_\lambda)^2}{2\tl y},
	\end{align}
with initial condition $r_\lambda(0)=\infty$.
Consider the hyperbolic angle $\tl\alpha_{\lambda}=\tl\alpha_{\lambda,\delta}$ between the points $\infty, \tl x+i \tl y, r_\lambda$, this is given by   $\tl\alpha_{\lambda}=2\arccot(\frac{\tl x-r_\lambda}{\tl y})$. More precisely, we can define a ``lifted'' version of this function on $\R$ that is continuous in $\lambda$ and $t$, satisfies $\tl \alpha_\lambda(0)=0$ and $\cot(\tl \alpha_\lambda/2)=\tfrac{\tl x-r_\lambda}{\tl y}$.

By It\^o's formula, together with a change of variable $\alpha_{\lambda}(t)  = \tl\alpha_{\lambda,\delta}(e^{-\beta t/4})$,  we get the SDE system
%	\begin{align*}
%	%	d(\frac{\tl x-r}{\tl y}) &= -\lambda\frac{\tl y^2+(\tl x-r)^2}{\tl y^2} dt+ \frac{4}{\beta(1-t)}(\Im\delta+(\Re\delta+1)\frac{\tl x-r}{\tl y})dt +\frac{2}{\sqrt{\beta(1-t)}}(dB_2-\frac{\tl x-r}{\tl y}dB_1)\\
%	d\alpha_\lambda &= \lambda dt+\frac{4}{\beta(1-t)}\Im[\delta(e^{-\im\alpha_\lambda}-1)] dt + \frac{2}{\sqrt{\beta(1-t)}}\Re[(e^{-\im \alpha_\lambda}-1)(dB_2+\im dB_1)].
%	\end{align*}
	\begin{align*}
	d\alpha_{\lambda} = \lambda \tfrac{\beta}{4}e^{-\frac{\beta}{4}t}dt +\Re[(e^{-\im \alpha_{\lambda}}-1)(dZ-\im\delta dt)],\quad\quad \alpha_{\lambda}(0)=0.
	\end{align*}

	Let $N(\lambda)$ be the right-continuous version of the limit $\lim\limits_{t\to\infty}\frac{1}{2\pi}\alpha_\lambda(t)$. It remains to prove that $N(\cdot)$ has the same distribution as the counting function of the spectrum of the $\Huaop$ operator. The proof relies on the oscillation theory of Dirac operators, see Section 4 in \cite{BVBV_op}, and it can be done exactly the same way as in Theorem 26 in \cite{BVBV_op}. The only ingredients that are needed to cover the general $\Re \delta+1/2>0$ case are the following: (1) the right endpoint of the $\Huaop$ operator is limit point if $\beta\le 4 (\Re \delta +1/2)$ and limit circle otherwise (see Proposition 31 in \cite{BVBV_op}), and (2) for $\beta\le 4 (\Re \delta +1/2)$ in the $\lambda>0$ case $\alpha_\lambda(t)$ converges to its limit from above a.s.~by Lemma \ref{lem:HP_alpha_sde}.
\end{proof}

\subsection{Proofs of the theorems related to $\Bessop$}

\begin{proof}[Proof of Theorem \ref{thm:zetaB}]
% 	Fix $\beta>0,a>-1$. Let $\zeta_B$ be the secular function of $\Bessop$. Note that $\intr \Bessop=0$ by symmetry. Recall the definition of secular function in \eqref{eq:zeta}, we have
% 	\begin{align*}
% 	\zeta_B(z)={\det}_2(I-z \,\ttr \Bessop) =\prod_{k} (1-z/ \,\lambda_k)e^{z/ \lambda_k},
% 	\end{align*}
% 	which is equivalent to the principal value product in \eqref{eq:Bess_prod}. 
	It will be more convenient to work with a time reversed version of the operator $\Bessop$. Let $y(u)=\exp(-\frac{\beta}{4}(2a+1)u+B(2u))$ and $\hat y(t)=y(u_\beta(t))$ with $u_\beta(t)=\frac{4}{\beta}\log t$. We consider the reversed Dirac operator
	\begin{align*}
		\tau_{\beta,a}^{\ttB} = \Dirop(i\hat y(t), \uu_0, \uu_1),\qquad t\in(0,1],
	\end{align*}	
	where $\uu_0 = [1,0]^t, \uu_1=[0,-1]^t$. Within this proof we use the simplified notation	$\tau_{\beta,a}$ for $\tau_{\beta,a}^{\ttB}$, and denote the secular function of $\tau_{\beta,a}$  by $\zeta_{\beta,a}$.
	By the symmetry of $\Bessop$, Lemmas \ref{lem:timerev} and \ref{lem:rotation}, we have
	\begin{align*}
	\rho J\tau_{\beta,a}J\rho^{-1} {\,\,\buildrel d\over =\,\,} \Bessop.
	\end{align*}
	Hence $\tau_{\beta,a}$ is orthogonal equivalent to $\Bessop$,  its eigenvalues have the same law of the $\BB_{\beta,a}$ process, and  $\zeta_{\beta, a}^{\ttB}$ has the same distribution as  $\zeta_{\beta,a}$.

	The statement about the Taylor expansion of $\zeta_{\beta,a}$
 follows from Proposition 9 in \cite{BVBV_szeta}, which shows that the $n$th Taylor coefficient of $\zeta_{\beta,a}$ can be evaluated using the multiple integral
	\begin{align*}
	 -\!\!\iiint\limits_{0<s_1<s_2<\cdots<s_n\le 1} \uu_0^t R(s_1)JR(s_2)J\cdots R(s_n)\uu_1 ds_1\cdots ds_n,\quad R(s) =\frac12\begin{pmatrix}
	\hat y(s)^{-1}& 0\\ 0 &\hat y(s)
	\end{pmatrix}.
	\end{align*}
Noting that the multiple integral is 0 when $n$ is odd, the statement about the Taylor expansion of $\zeta_{\beta,a}$ follows. 
 
The SDE representation of $\zeta_{\beta,a}$ can be shown similarly as the analogue statement for $\zeta^{\ttHP}_{\beta, \delta}$.
	By Proposition  13 in \cite{BVBV_szeta}, we have $\zeta_{\beta,a}(z) = [1,0]H(1,z)$, where $H:(0,1]\times \CC\mapsto\CC^2$ is the unique function that solves the ODE
	\begin{align*}
	J\frac{d}{dt}H(t,z) = zR(t)H(t,z),\qquad \lim_{t\to 0}H(t,z) = \uu_0.
	\end{align*}
	Introduce $X_u = \begin{pmatrix}1 &0\\0&y(u)\end{pmatrix},u\le 0$. Then we have $\zeta_{\beta,a}(z)=[1,0]\cH_0(z)$ where $\cH_u(z)=X_uH(e^{\frac{\beta}{4}u},z)$. The fact that $\cH$ satisfies the SDE \eqref{eq:B_cH}  can be checked using It\^o's formula and an adaptation of the approximating scheme described in Propositions 20 and 43 in \cite{BVBV_szeta}. 

Note that the Taylor coefficients of $\zeta_{\beta, a}$ can also be expressed by differentiating the SDE \eqref{eq:B_cH} and solving the resulting system of SDEs. This gives another way to derive \eqref{eq:B_Taylor}.
% Set $\cH_u=[\cA_u,\cB_u]^t$, and denote the Taylor coefficients of $\cA, \cB$ at $z=0$ by $\cA_u^{(n)}, \cB_u^{(n)}$. 
% Using the same approach as in the proof of Theorem \ref{thm:zetaHP} we get the following SDE system for $\cA_u^{(n)}, \cB_u^{(n)}$:
% \begin{align*}
% d\cA^{(n)}&= \tfrac{\beta}{8}e^{\frac{\beta}{4}u} \cB^{(n-1)}du, \\
% d\cB^{(n)}&=\sqrt{2}\cB^{(n)}dB+(1-\tfrac{\beta}{4}(2a+1))\cB^{(n)}du-\tfrac{\beta}{8}e^{\frac{\beta}{4}u} \cA^{(n-1)}du,
% \end{align*}
% with $\cA_u^{(0)}=1$, $\cB_u^{(0)}=0$. Mimicking the proof of Propositions 45 and 47 in \cite{BVBV_szeta} one can prove that the solution of the above system exist, and it is given by equations \eqref{eq:B_taylor_1}, \eqref{eq:B_taylor_2}.
\end{proof}

\noindent\textbf{Acknowledgements.} 
The authors thank  B\'alint Vir\'ag for valuable discussions.  B.~V.~was partially supported by the NSF award DMS-1712551.

\def\cprime{$'$}

{{  
\bigskip
  \footnotesize

\noindent Yun Li, \textsc{Department of Mathematics, University of Wisconsin -- Madison, 480 Lincoln Dr,
Madison, WI 53706},  \texttt{li724@wisc.edu}

  \medskip

\noindent Benedek Valk\'o, \textsc{Department of Mathematics, University of Wisconsin -- Madison, 480 Lincoln Dr,
Madison, WI 53706}, \texttt{valko@math.wisc.edu}

}}

\end{document}